\documentclass[11pt]{article}
\usepackage{amsmath}
\usepackage{amssymb}
\usepackage{latexsym}
\usepackage{amsthm}
\usepackage{fullpage}
\usepackage{graphicx}
\usepackage{color}
\usepackage[colorinlistoftodos]{todonotes}
\definecolor{ao(english)}{rgb}{0.0, 0.5, 0.0}
\usepackage{marginnote}
 \numberwithin{equation}{section}
\newtheorem{example}{Example}[section]
\newtheorem{theorem}{Theorem}[section]
\newtheorem{lemma}{Lemma}[section]

\newtheorem{proposition}{Proposition}[section]

\newtheorem{remark}[example]{Remark}
\newtheorem{definition}[example]{Definition}

\definecolor{Green}{RGB}{0,0, 0}

\newcommand{\R}{{\mathbb R}}
\newcommand{\N}{{\mathbb N}}

\newcommand{\be}{\begin{eqnarray}}

\newcommand{\ee}{\end{eqnarray}}

\renewcommand{\d}{{\rm d}}

\newcommand{\md}{{\rm d}}

\renewcommand{\O}{\Omega}

\newcommand{\A}[1]{\langle#1\rangle}

\newcommand{\rca}{{\mathcal{M}}}

\newcommand{\wto}{\rightharpoonup}

\renewcommand{\wto}{\rightharpoonup}

\newcommand{\cof}{{\rm Cof }\,}
\newcommand{\GY}{\mathcal{GY}^{\infty}(\O;\R^{n\times n})}

\newcommand{\wY}{{\stackrel{Y}{\wto}}}
\newcommand{\wstar}{{\stackrel{*}{\wto}}}
\renewcommand{\det}{{\rm det}\,}
\newcommand{\FIRST}{\color{black}}
\newcommand{\SECOND}{\color{black}}
\newcommand{\WE}{\color{black}}
\newcommand{\EOR}{\color{black}}
\newcommand{\MAY}{\color{black}}

\author{Barbora Bene\v{s}ov\'{a}\footnote{Institute of Mathematics,
University of W\"{u}rzburg, Emil-Fischer-Stra\ss e 40, 97074 W\"{u}rzburg,    Germany
 } \  \ Martin Kru\v{z}\'{\i}k\footnote{The Czech Academy of Sciences, Institute of Information Theory and Automation,
 Pod vod\'{a}renskou
v\v{e}\v{z}\'{\i}~4, CZ-182~08~Praha~8, Czech Republic (corresponding
address) \& Faculty of Civil Engineering, Czech Technical
University, Th\'{a}kurova 7, CZ-166~ 29~Praha~6, Czech Republic}\ \ Anja Schl\"{o}merkemper$^*$
}
\title{A note on locking materials and gradient polyconvexity}
\begin{document}
\maketitle
\begin{abstract}
\WE
We use gradient Young measures generated by Lipschitz maps to define a relaxation of integral functionals  which are allowed to attain the value $+\infty$ and can model ideal locking  in elasticity  as defined by Prager in 1957. % \cite{prager}.  
Furthermore, we show the existence of minimizers for variational problems for elastic materials with energy densities that can be expressed in terms of a function being continuous in the deformation gradient and convex in the gradient of the cofactor (and possibly also the gradient of the determinant) of the corresponding deformation gradient. We call the related energy functional gradient polyconvex. Thus, instead of considering second derivatives of the deformation gradient as in second-grade materials, only a weaker higher integrability is imposed. Although the second-order gradient of the deformation is not included in our model, gradient polyconvex functionals allow for an implicit uniform positive lower bound on the determinant of the deformation gradient on the closure of the domain representing the elastic body.  Consequently, the material does not allow for  extreme local compression. 
\end{abstract}
%\todo[inline]{ Items (17) and (28) of the first report done. Please, keep the convention that changes related to the first report are \FIRST blue \EOR, to the second \SECOND red \EOR and our changes are \WE  purple \EOR}
\section{Introduction}
Modern mathematical theory of nonlinear elasticity typically assumes that the first Piola-Kirchhoff stress tensor has a potential, the so-called stored energy density $W\ge 0$. Materials fulfilling this assumption are referred to as \emph{hyperelastic} materials.

The state of the hyperelastic material is described by its deformation $y:\O \to \R^n$ which is a mapping that assigns to each point in the reference configuration $\Omega$ its position after deformation. In what follows,  we assume that $\O\subset\R^n$ (usually $n=2$ or $n=3$) is a bounded Lipschitz domain. Of course, the deformation need not be the only descriptor of the state; it can  additionally be described by the temperature, inner variables etc. Nevertheless, in this article, we will not consider such more general situations. 

Stable states of specimen are then found by minimizing the energy functional
\begin{align}
I(y):=\int_\O W(\nabla y(x))\,\md x-\ell(y)
\label{energy-funct}
\end{align} 
over a class of admissible deformations $y:\overline\O\to\R^n$.  Here $\ell$ is a linear bounded functional on the set of deformations expressing the work of external loads on the specimen and $\nabla y$ is the deformation gradient which quantifies the strain. Let us note that the elastic energy density in \eqref{energy-funct} depends on the first gradient of $y$ only, which is the simplest and canonical choice. Nevertheless, $W$ might depend also on higher gradients of $y$ for so-called non-simple materials. Also, various other energy  contributions representing the work of external forces can  be included; we will include some of them for the mathematical study later.

The principle of frame-indifference requires that $W$ satisfies for all $F\in\R^{n\times n}$ and all proper rotations $R\in{\rm SO}(n)$ that 
\begin{align}\label{frameindiff}
W(F)=W(RF)\  .
\end{align}

%Formally, calculating the Euler-Lagrange equation of \eqref{energy-funct} gives the balance of momentum featuring the first Piola-Kirchhoff stress tensor. However, the balance of momentum might be satisfied also for deformations that are not minimizers of \eqref{energy-funct}, but these states are not necessarily local minimizers of $I$. Therefore minimizing \eqref{energy-funct} is also justified from the point of view of physics. On the top of that, it allows us to 
%use tools from the calculus of variations to handle the problem. Let us however refer to \cite{ball-EL} on more details regarding the connection between the Euler-Lagrange equation and the minimization problem \eqref{energy-funct}.  

From the applied analysis point of view, an important question is for which stored energy densities the functional $I$ in \eqref{energy-funct} possesses minimizers. Relying on the direct method of the calculus of variations, the usual approach to address this question is to study \emph{(weak) lower semi-continuity} of the functional $I$ on appropriate Banach spaces containing the admissible deformations. See e.g.~\cite{dacorogna} or the recent review \cite{benesova-kruzik-wlsc} for a detailed exposition of weak lower semicontinuity.  

Functionals that are not weakly lower semicontinuous might still possess minimizers in some specific situations, but, in general, existence of minimizers can fail. From the point of view of  materials science, such a setting can correspond to the formation of microstructure of strain-states; as it is found in, for example, shape-memory alloys \cite{ball-james, bhattacharya, mueller}.  A generally accepted modeling approach for such materials is to calculate the (weakly) lower semicontinuous envelope of $I$, the so-called \emph{relaxation}, see, e.g., \cite{dacorogna}. %Let us remark that the knowledge of the relaxation of $I$ is also crucial in calculating the $\Gamma$-limit \cite{braides, dalmaso}, which can be advantageously used for discrete-to-continuum transitions that may serve as a justification macroscopic models from the microscopic point of view. 
Thus, next to the characterization of weak lower semicontinuity also the calculation of the lower semicontinuous envelope is of interest in the calculus of variations.
%the calculation of the lower semicontinuous envelope is equally important as the characterization of weak lower semicontinuity itself. 

A characterization of weak lower semicontinuity of $I$ is standardly available  if $W$ is of $p$-growth; that is,
for some $c>1$, $p \in (1,+\infty)$ and all $F\in \R^{n \times n}$ the inequality
\begin{equation}
\frac1c(|F|^p-1) \leq W(F) \leq c(1+|F|^p)
\label{pGrowth}
\end{equation}
is satisfied, which in particular implies that $W<+\infty$. Indeed, in this case, the natural class for admissible deformations is the Sobolev space $W^{1,p}(\Omega; \R^n)$ and it is well known that the relevant condition is the quasiconvexity of $W$ (see Section \ref{section-prelim} formula \eqref{quasiconvexity} for a definition) which is then equivalent to weak lower semicontinuity of $I$ on $W^{1,p}(\Omega; \R^n)$. If $W$ is not quasiconvex, then the relaxation can be computed by replacing $W$ by its quasiconvex envelope, i.e., the supremum of all quasiconvex functions lying below $W$, see, e.g., \cite{dacorogna}.

Quasiconvexity turns out to be an equivalent condition also  for weak*-lower semicontinuity on $W^{1,\infty}(\Omega; \R^n)$. If one wants to consider admissible deformations in the class of Lipschitz functions then this can be guaranteed by the following coercivity of the stored energy function:
\begin{equation}
W(F)\begin{cases}
<+\infty \quad\text{ if $|F|\le\varrho$}\\
= +\infty \quad \text{ if $|F|>\varrho$} ,
\end{cases}
\label{locking-energy}
\end{equation}
for some $\varrho>0$. This corresponds to a material model for which  the region of elasticity is given by a closed ball $\overline{B(0,\varrho)}:=\{F\in\R^{n\times n}:\ |F|\le\varrho\}$. For larger strains, the elasticity regime is left and a more elaborate model, corresponding to e.g.\ plasticity, damage etc.\ has to be employed. 

A similar concept, motivated by \emph{material locking}, was introduced by Prager in \cite{prager}, see also \cite{ciarlet-necas,demengel-suquet,golay-seppecher,panagiotopoulos,phillips,schuricht} for newer results. According to Prager's classification of elastic materials, a material is called elastically {\it hard} if its elastic constants increase with the increasing strain.  Perfectly (or ideally) locking materials extrapolate this property by assuming that  the material gets locked  (i.e.\ becomes stiff or rigid), once  some strain measure  reaches a prescribed value. (An analogous recent  concept is called  ``strain-limiting materials''; cf.~\cite{rajagopal}, where the elastic strain is bounded independently of the applied stress.)  Prager \cite{prager} introduced a locking constraint in the form $L(\nabla y)\le 0$ almost everywhere in $\O$ with
$$L(F):=\left|\frac12(F+F^\top)+(1-\frac23({\rm tr}\, F)){\rm Id}\right|^2 -\varrho ,$$ where ${\rm Id}$ is the identity matrix, ``tr'' denotes the trace,  and  $\varrho>0$ is a material parameter.  This function is, however, not suitable for nonlinear elasticity because it is not frame-indifferent, i.e., \eqref{frameindiff} is not satisfied with $L$ instead of $W$.
 Ciarlet and Ne\v{c}as \cite{ciarlet-necas} removed this issue by setting 
\begin{align}\label{fi-locking} L(F):=\frac14\left|F^\top F-\frac{|F|^2}3{\rm Id}\right|^2-\varrho .\end{align} 
Nevertheless, $L$ in \eqref{fi-locking} is  not convex.  As convexity is needed for the relaxation result in Section~\ref{sec-relaxation} we will  work with the following locking constraint, which is convex and frame-indifferent. 
\begin{align}\label{locking}
L(F):=|F|-\varrho.
\end{align} 
Notice that \eqref{locking-energy} can be replaced by  assuming that $W$ is finite only if $L(F)\le 0$ where $L$ corresponds to \eqref{locking}. 
A suitable choice of $|\cdot|$ allows us to restrict the deformation in the desired components of the particular strain measure (e.g.,  the Cauchy-Green strain tensor $(\nabla y)^\top\nabla y$, or the deformation gradient $\nabla y$)  by  requiring that $L(\nabla y)\le 0$ a.e.~in $\O$. Note that the pointwise character of locking constraints allows us to control locally  the strain appearing in the material. 

We emphasize that using a model requiring \eqref{locking-energy} (or \eqref{locking}) does not mean that deformations of the material with $|F|>\varrho$ are not possible in general. It just means that such deformation cannot be purely elastic, but must have inelastic parts, as well. In other words, one should use  physically richer models
%involving not only elasticity but also plasticity, damage, etc., 
to describe a modeled experiment. Altogether, locking constraints can serve as criteria whether or not we are authorized to use merely an elastic description of the material behavior. 

The locking constraint \SECOND $L\leq 0$ with $L$ as in \EOR \eqref{locking} models the fact that once the strain gets too large, the material leaves the elastic regime under strong tension. Of course, any elastic material will also resist compression, which is usually modeled by assuming
\begin{align}
\label{det} 
W(F)\to+\infty \text{ if }\det F\to 0_+ . 
\end{align}
The property in \eqref{det} is represented in form of a ``soft'' constraint. However, it could also be replaced by a ``hard'' locking constraint, similarly to the locking theory above: Let
\begin{equation}
\label{locking-det}
L(F):=\varepsilon - \det F,
\end{equation}
for some $\varepsilon > 0$ and assume that $W$ in \eqref{energy-funct} is finite only if $L\leq 0$. Like above, the threshold $\varepsilon$ models the ``boundary'' of the elastic region beyond which a purely elastic model is not applicable. We refer to \cite{fosdick} for a treatment of this constraint in linearized elasticity.

In Section 3, we study relaxation under the constraint \SECOND $L(\nabla y) \leq 0$ with $L$ as in \eqref{locking} and elastic boundary conditions (see \eqref{functionalJ}) \EOR by means of gradient Young measures. Indeed, while the vast majority of relaxation techniques available in the literature concern only energies that take finite values, the only relaxation result under the constraint based on \eqref{locking}, to the best of the authors' knowledge, is due to Wagner \cite{wagner-rel} (see also \cite{wagner,wagner-env}), who characterized the relaxed energy by means of an infimum formula. \FIRST We also refer to \cite{carbone, zappale} for relaxation results of unbounded functionals with scalar-valued competing maps and to  \cite{champion,gloria} for homogenization problems for unbounded functionals. \EOR However, the proof via Young measures, provided here, is considerably simpler and perhaps sheds more light on the difficulties when handling locking constraints. Indeed, the biggest difficulty, that we have to cope with, is that we have to prove that for any $y \in W^{1,\infty}(\Omega;\R^n)$ there exists a sequence $\{y_k\}_{k\in\N} \subset W^{1,\infty}(\Omega;\R^n)$ weakly* converging to $y$ such that 
$$
\overline{I}(y) = \lim_{k \to +\infty} I(y_k),
$$
with $\overline{I}(y)$ the relaxation of $I$. Following the standard methods (see e.g.\ \cite{dacorogna}) it could happen that $\{y_k\}_{k\in\N}$ does not satisfy the locking constraint even if $y$ does. We resolve this issue by a careful scaling and continuity of $W$ on its domain. This idea first appeared in \cite{k-p1} and was similarly used in \cite{wagner-rel}, too.

% As an application for the usage of the relaxation, we prove a discrete-to-continuum transition for materials satisfying the locking constraint \eqref{locking}. Here, additionally, we need to be careful in selecting the recovery sequence suitably in such a way that it complies with \eqref{locking}.

As far as the constraint \SECOND $L(\nabla y) \leq 0$ with $L$ as in \EOR \eqref{locking-det} is concerned, the situation is even less explored. Indeed, the study of (weak) lower semicontinuity of energies with a density that is infinite for $\det F \le 0$ and satisfies \eqref{det} is mostly inaccessible with the present methods of the calculus of variations. For example, it remains open to date if \FIRST $I$  from \eqref{energy-funct} with an energy density $W$  that is additionally quasiconvex possesses minimizers \cite[Problem~1]{ball-puzzles}. \EOR Only scattered results in particular situations have been obtained \cite{benesova-kamp,benesova-kruzik,krw}, see also \cite[Section 7]{benesova-kruzik-wlsc} for a review. For a related work that involves a passage from discrete to continuous systems and dimension reduction as well as constraints on the determinant we refer to \cite{lazzaroni-palombaro-schloemerkemper}, see also \cite{ALP}.\\

In Section~\ref{lock-det}, we prove that \FIRST the energy functional $I$ from \eqref{energy-funct} with a quasiconvex  stored energy density $W$ \EOR satisfying the locking constraint \SECOND $L(\nabla y) \leq 0$ with $L$ as in \EOR \eqref{locking-det} indeed \emph{has a minimizer}. 
%The proofs are based on fine extension properties of deformation mappings or on convex integration. 
Nevertheless, a relaxation result remains out of reach. We refer to \cite{conti-dolzmann} for a partial relaxation result reflecting \eqref{det} but requiring that the lower semicontinuous envelope  \FIRST of $W$ \EOR is polyconvex. 
We recall that $W:\R^{n\times n}\to\R\cup\{+\infty\}$ is polyconvex \cite{ball77} if we can write for all $A\in\R^{n\times n}$ that 
$W(A)=h(T(A))$ where $T(A)$ is the vector of all minors (subdeterminants) of $A$ and $h$ is a convex \WE and lower semicontinuous function.\EOR\\

Deformations that satisfy locking constraints naturally appear in the study of so-called non-simple materials. For such materials, the energy depends not only on the first gradient of the deformation but also on higher gradients; in particular, the second one. Such models were introduced by Toupin \cite{Toupin:62,Toupin:64} and further developed by many researchers, see e.g.~\cite{BCO,vidoli,forest,mielke-roubicek,podio} for physical background and mathematical treatment in versatile context including elastoplasticity and damage. The contribution of the higher gradient is usually associated to interfacial energies, as in  e.g.~\cite{ball-crooks,ball-mora,mielke-roubicek,podio,Silh88PTNB} which work with an energy functional of the type
\begin{equation}
J(y)= \int_\Omega (w(\nabla y(x)) + \gamma |\nabla^2 y(x)|^d) \mathrm{d} x,
\label{second-order}
\end{equation}
for some \FIRST $\gamma > 0$ \EOR and $d>1$. Now, if $d>n$, any deformation of finite energy will \SECOND satisfy $L(F) \leq 0$ with $L$ from \eqref{locking} and $\varrho$ \EOR depending only on the energy bound by Sobolev embedding.
Actually, Healey and Kr\"omer \cite{healey-kroemer} showed that in this situation,  if $w$ is suitably coercive in the inverse of the Jacobian of the deformation, the locking constraint based on \eqref{locking-det} is satisfied, too. This allows to show that minimizers of the elastic energy satisfy a weak form of the corresponding Euler-Lagrange equations. To prove the lower bound on the determinant, they exploit  that $\det\nabla y$ is H\"{o}lder continuous in $\overline{\O}$. 

Nevertheless, the form of the contribution containing the second gradient in \eqref{second-order} seems to be motivated mostly by its mathematical simplicity.

In Section~\ref{gradient-polyconvexity}, we show that, at least as far as existence of solutions \SECOND as well as the above mentioned locking constraints  are \EOR concerned, the contribution of the whole second gradient is not needed. Indeed, we introduce the notion of {\it gradient polyconvexity} where we consider  \MAY energy functionals \EOR\FIRST with an energy density \EOR that can be expressed in terms of a function which, if $n=3$, is convex in the gradient of the cofactor matrix of the deformation gradient as well as in the gradient of the determinant of the deformation gradient. In general dimensions we may consider energies that can be expressed in terms of a function which is convex in the gradient of the minors of the order $n-1$. This new type of functionals involving higher derivatives allows for the following interpretation in three dimensions: Since the determinant is a measure of the transformation of volumes and the cofactor measures the transformation of surfaces in a material, cf.\ e.g.\ \cite[Theorem~1.7--1]{ciarlet}, the maximal possible change thereof is controlled by letting the energy depend on the gradients of these measures. 

\WE We prove existence of minimizers  for such materials relying on the weak continuity of minors, similarly as in classical polyconvexity due to J.M. Ball \cite{ball77}.  First, we prove existence of minimizers for gradient polyconvex functionals in the special case that the energy density depends on the deformation gradient and the gradient of the cofactor of the deformation gradient (but not on the gradient of the determinant of the deformation gradient), cf.~Proposition~\ref{prop-grad-poly}; as we show in Proposition~\ref{prop-grad-poly3} the energy density may depend also on the spatial variable, the deformation and the inverse of the deformation gradient. Secondly, we consider gradient polyconvex functionals whose energy density depends in a convex way on the gradient of the cofactor matrix as well as on the gradient of the determinant of the deformation gradient, cf.\ Proposition~\ref{prop-grad-poly2}. \EOR

\SECOND We point out that the setting of gradient polyconvex energies allows to incorporate quite general locking constraints of the type $L(\nabla y) \leq 0$ a.e.\ for some lower semicontinuous $L$, see Propositions~\ref{prop-grad-poly}, \ref{prop-grad-poly3} and \ref{prop-grad-poly2}. Finally,  by following the lines of a locking result by Healey and Kr\"omer \cite{healey-kroemer}, we show that any minimizer satisfies that $\det(\nabla y) \geq \varepsilon$ a.e.\ for some $\varepsilon > 0$, see Propositions~\ref{prop-grad-poly} and \ref{prop-grad-poly2}. Hence the related elastic systems are prevented from full compression. While we have to assume that the Sobolev index for the cofactor matrix of the deformation gradient is larger than 3 in the first setting (Proposition~\ref{prop-grad-poly}), it turns out that this does not have to be assumed in the second setting (Proposition~\ref{prop-grad-poly2}), where, however, we have to assume that the Sobolev index for the determinant of the deformation gradient is larger than 3.\EOR 

\bigskip

\noindent To summarize our results, within this article we 
\begin{itemize}
\item prove a relaxation result under the locking constraint based on \eqref{locking} for large deformation gradients using the Young measure representation, see Section~\ref{sec-relaxation};
% \item prove a discrete-to-continuum transition under the constraint \eqref{locking} relying on the already available relaxation
\item prove existence of minimizers for energies satisfying the constraint based on \eqref{locking-det}, which prevents the material from extreme local compression, see Section~\ref{lock-det};
\item introduce the notion of gradient polyconvex energies and provide a related result on existence of minimizers; further, we observe that  admissible deformations need to fulfill \eqref{locking-det}, see Section~\ref{gradient-polyconvexity}. Moreover,  every locking constraint $L(\nabla y)\le 0$ is admitted provided $L:\R^{3\times 3}\to\R$ is lower semicontinuous. 
\end{itemize}

\section{Preliminaries}
\label{section-prelim}

The relaxation results in this contribution are proved by employing so-called \emph{gradient Young measures}. Thus, we  recall some known results together with the necessary notation and refer to \cite{pedregal, r} for an introduction. We shall be working with functions in Lebesgue or Sobolev spaces over a bounded Lipschitz domain $\Omega$ with values in $\R^m$ denoted, as is standard, by $L^p(\Omega;\R^m)$ and $W^{1,p}(\Omega;\R^m)$, $1\le p\le +\infty$, respectively. Continuous  functions over a set $O$ and values in $\R^m$ are denoted by $C(O; \R^m)$.  If $m=1$ we omit the range in the function spaces. Finally, $\mathcal{M}(\R^{n \times n})$ denotes the set of all Radon measures on $\R^{n \times n}$. Let us remind that, by the Riesz theorem,
$\mathcal{M}(\R^{n\times n})$, normed by the total variation, is a Banach space which is
isometrically isomorphic with $C_0(\R^{n\times n})^*$, the dual of $C_0(\R^{n\times n})$. Here $C_0(\R^{n\times n})$ stands for
the space of all continuous functions $\R^{n\times n}\to\R$ vanishing at infinity. Further, %$C(M)$ denotes the space of continuous functions defined on the set $M$, and 
$\mathcal{L}^n$ and $\mathcal{H}^n$ denote the $n$-dimensional Lebesgue and  Hausdorff measure, respectively. As to the matrix norm on $\R^{n\times n}$ we consider the Frobenius one, $|F|^2:=\sum_{i,j=1}^nF_{ij}^2$ Analogously, \MAY the Frobenius norm for $F\in\R^{n\times n\times n}$ is defined \EOR as $|F|^2:=\sum_{i,j,k=1}^nF_{ijk}^2$.  \WE \MAY The functional \EOR  $G:\R^N\to\R\cup\{+\infty\}$ is continuous (lower semicontinuous) if $X_k\to X$ in $\R^N$ for $k\to+\infty$ implies that $\lim_{k\to+\infty}G(X_k)= G(X)$ ($\liminf_{k\to+\infty}G(X_k)\ge G(X)$).
Finally, let us recall that ``$\wto$'' denotes the weak convergence in various Banach spaces. 
\EOR 

\hspace{2ex}

\noindent {\bf Young measures.} Young measures characterize the asymptotic behavior of non-linear functionals along sequences of rapidly oscillating functions. It is well known that fast oscillation in a sequence of functions $\{Y_k\}_{k\in\N} \subset L^\infty(\Omega;\R^{n \times n})$ can cause failure of strong convergence of this sequence but, relying on the Banach-Alaouglu theorem, \FIRST it is  still \EOR possible to assure that a (non-relabelled) subsequence  $\{Y_k\}_{k\in\N}$ converges weakly$^*$ to $Y \in L^\infty(\Omega;\R^{n \times n})$. In such a case, the sequence $\{f(Y_k)\}_{k\in \N}$ for a continuous function $f:\R^{n \times n}\to \R$ is bounded in $L^\infty(\Omega)$ so that, at least for a non-relabeled subsequence, it converges weakly$^*$ in $L^\infty(\Omega)$. It is clear, however, that the knowledge of the weak$^*$ limit $Y$ is not sufficient to characterize $\mathrm{w^*{-}lim}_{k\to+\infty} f(Y_k)$. A reason for this is that the weak limit simply does not retain enough information about the oscillating sequence; in particular, the weak limit can be understood as some ``mean value'' of the oscillations but it does not record any further properties apart from this ``average''. Yet, without further knowledge, a limit passage under non-linearities is not possible in general.

Young measures provide a tool how to record more information on the oscillations in the weakly$^*$ converging sequence. Roughly speaking, they also store information on ``between which values'' and with ``which weight'' the functions in the sequence oscillate. In particular, the \emph{fundamental theorem on Young measures} \cite{y} states that for every sequence $\{Y_k\}_{k\in\N}$
bounded in $L^\infty(\O;\R^{n\times n})$ 
%such that $Y_k(x)\in S$ t
there exists a subsequence $\{Y_k\}_{k\in\N}$ (denoted by
the same indices for notational simplicity) and a family of probability measures 
$\nu=\{\nu_x\}_{x\in\O}$ such that for all $f \in C(\R^{n \times n})$
\begin{align}\label{young-measures}
\lim_{k\to+\infty}\int_\Omega \xi(x) f(Y_k) \md x = \int_\Omega \int_{\R^{n\times n}}\xi(x) f(A)\nu_x(\d A) \md x,
\end{align}
for all $\xi \in L^1(\Omega).$ \FIRST
The obtained family of probability measures $\nu=\{\nu_x\}_{x\in\O}$ is referred to as a \emph{Young measure} and $\{Y_k\}_{k \in \N}$ is its \emph{generating sequence}. Alternatively, we say that $\{Y_k\}$ generates $\nu$. Let us denote the set of all Young measures by ${\cal Y}^\infty(\O;\R^{n\times n})$. Then ${\cal Y}^\infty(\O;\R^{n\times n}) \subset L^\infty_{\rm w}(\O;\rca(\R^{n\times n}))\cong L^1(\O;C_0(\R^{n\times n}))^*$; here $L^\infty_{\rm w}(\O;\rca(\R^{n\times n}))$ is the set of essentially bounded, weakly* measurable\footnote{The adjective ``weakly* measurable'' means that,
for any $f\in C_0(\R^{n\times n})$, the mapping
$\O\to\R$ such that $x\mapsto\A{\nu_x,f}=\int_{\R^{n\times n}} f(A)\nu_x(\d A)$ is
measurable in the usual sense.} mappings $\O\to \mathcal{M}(\R^{n\times n})$ such that
$x\mapsto\nu_x$. Moreover, as Young measures take values only in probability measures, it
is known (see e.g.~\cite[Lemma~3.1.5]{r})  that ${\cal Y}^\infty(\O;\R^{n\times n})$ is a convex subset of $L^\infty_{\rm w}(\O;\mathcal{M}(\R^{n\times n}))$.
\EOR   

We have the following characterization: If a measure $\mu=\{\mu_x\}_{x\in\O}$ is supported in a compact set $S \subset \R^{n \times n}$ for almost all $x\in\O$ and $x\mapsto\mu_x$ is weakly* measurable then there exists a sequence  $\{Z_k\}_{k\in\N}\subset L^\infty(\O;\R^{n\times n})$, with $Z_k(x)\in S$ for a.e.\ $x\in \Omega$, such that \eqref{young-measures} holds with $\mu$ and $Z_k$ instead of $\nu$ and $Y_k$, respectively. \SECOND On the other hand, we have that every measure $\nu_x \in \mathcal{Y}^\infty(\Omega; \R^{n \times n})$ \EOR (generated by the sequence $\{Y_k\}_{k\in\N}$) is supported on the set $\bigcap_{l=1}^{+\infty}\overline{\{Y_k(x);\ k\ge l\}}$ for almost all $x \in \Omega$; cf.~\SECOND \cite[Theorem~I]{balder}, \cite[Proposition 5]{valadier} \EOR.
We define the support of $\nu$ as 
$$\text{supp}\,\nu:=\bigcup_{\text{a.e.}\,x\in\O} \text{supp}\,\nu_x\ $$
and, for almost all $x\in\O$,  the first moment of the Young measure $\nu$ as
$$\bar\nu(x):=\int_{\R^{n\times n}} A\nu_x(\md A).$$

% In what follows, $\regn$ shall denote the set of invertible matrices in $\R^{n\times n}$.  Further, we define the following subsets of the set of invertible matrices:
% \begin{align}
% \label{Rvarrho} \Rrhon&:=\{A\in\regn;\ \max(|A|,|A^{-1}|)\le\varrho\} , 
% \end{align}
% for $1\le\varrho<\infty$.
\bigskip

\noindent
{\bf Gradient Young measures.} An important question, namely which Young measures are  generated by sequences of gradients of Sobolev functions (called \emph{gradient Young measures}),  was answered by Kinderlehrer and Pedregal in \cite{k-p1,k-p}.
We recall their result for Young measures generated by gradients for sequences bounded in   $W^{1,\infty}(\O;\R^n)$. The set of such measures is denoted by $\GY$.  

\begin{theorem}[due to \cite{k-p1}]\label{qc-measures-characterization}
A Young measure $\nu=\{\nu_x\}_{x\in\O} \in \mathcal{Y}(\Omega; \R^{n \times n})$ is a gradient Young measure, i.e., it  belongs to  $\GY$ if and only if the following three conditions are satisfied simultaneously: \\
\noindent
(i) there is a compact set $\mathcal{S}\in \R^{n\times n}$ such that $\text{supp}\,\nu\subset \mathcal{S}$ ,\\

\noindent
(ii) there is $y\in W^{1,\infty}(\O;\R^n)$ such that for almost all  $x\in\O$
\begin{align}\label{fm}
\nabla y(x)= \SECOND \bar{\nu}_x \EOR = \int_{\R^{n\times n}} A\nu_x(\md A) ,
\end{align}
\noindent
(iii)  there is $\mathcal{N}\subset\O$ of  zero Lebesgue measure such that for $y$ from \eqref{fm} and all quasiconvex functions $f:\R^{n\times n}\to\R$, it holds that for   all  $x\in\O\setminus\mathcal{N}$
\begin{align}\label{jensen-qc}
f(\nabla y(x))\le \int_{\R^{n\times n}} f(A)\nu_x(\md A) .
\end{align}
\end{theorem}

We recall that $f:\R^{n\times n}\to\R$ is said to be quasiconvex (in the sense of Morrey \cite{morrey}) if 
\begin{align}\label{quasiconvexity}f(A)\mathcal{L}^n(\O)\le\int_\O f(\nabla\varphi(x))\,\md x\end{align} 
for all $A\in\R^{n\times n}$ and all $\varphi\in W^{1,\infty}(\O;\R^n)$ such that $\varphi(x) = Ax$ on $\partial \Omega$. Condition (ii) in Theorem \ref{qc-measures-characterization} says that the first moment $\bar\nu$ of the Young measure $\nu$ is $\nabla y$, while (iii) is a Jensen-type inequality for quasiconvex functions. 

We introduce the so-called {\em $Y$-convergence}; i.e., the convergence of a generating sequence toward a Young measure. 

\begin{definition}\label{convergence}
Assume that $\{y_k\}_{k\in\N}\subset W^{1,\infty}(\O;\R^n)$. Then
$y_k\stackrel{Y}{\wto} (y,\nu)\in C(\overline\O;\R^n)\times\GY$ as $k\to+\infty$ if 
$y_k\to y$ in $C(\overline\O;\R^n)$ and $\{\nabla y_k\}$ generates $\nu$. This convergence is called
the $Y$-convergence.
\end{definition}

The following statement was proved by M\"{u}ller in \cite{mueller1} and it is a generalization of a former result due to Zhang \cite{zhang}. \FIRST It says that a Young measure supported on a convex bounded set can be generated by a sequence of gradients of Sobolev functions taking values in an arbitrarily small neighborhood of this set. \EOR  Before giving the precise statement of this result, we recall that dist$(A,\mathcal{S}):=\inf_{F\in \mathcal{S}}|A-F|$ for $A\in\R^{n\times n}$, $\mathcal{S}\subset \R^{n\times n}$.

\begin{proposition}\label{mueller-proposition}
Let $\nu\in\GY$ and let $y$ a $W^{1,\infty}(\Omega;\R^n)$-function defined through its mean value, i.e., found via \eqref{fm}. Finally, assume that $\mathcal{S}$ in Theorem~\ref{qc-measures-characterization} is convex. Then there is 
$\{y_k\}_{k\in\N}\subset W^{1,\infty}(\O;\R^n)$ such that $y_k\stackrel{Y}{\wto} (y,\nu)\in C(\overline\O;\R^n)\times\GY$ 
and $\mathrm{dist}(\nabla y_k,\mathcal{S})\to 0$ in $L^\infty(\O;\R^{n\times n})$ as $k\to+\infty$.
\end{proposition}

%\todo[inline]{Maybe put also something on $\Gamma$-convergence?}

\bigskip

\bigskip

\section{Relaxation under locking constraints \SECOND based on \eqref{locking}\EOR}
\label{sec-relaxation}
In this section, we prove a relaxation result for functionals involving a stored energy density  $W\in C(\overline{B(0,\varrho)})$ and a locking constraint $L$ from \eqref{locking}. 

In more detail, we seek to solve
\begin{align}\label{functionalJ} 
&\text{minimize} && J(y):=\int_\O W(\nabla y(x))\,\md x-\ell(y) +\alpha\|y-y_0\|_{L^2(\Gamma;\R^n)} ,\nonumber\\
&\text{subject to} &&y\in \SECOND \mathcal{A}_\varrho ,\ \ \end{align}
with
\SECOND
\begin{align} \label{**}
\mathcal{A}_\varrho:=\{y\in W^{1,\infty}(\O;\R^n);\ \|\nabla y\|_{L^\infty(\O;\R^{n\times n})}\le \varrho\}.
\end{align}
\EOR

Here, $\varrho > 0$ is the constant introduced in \eqref{locking}, $\Gamma\subset\partial\O$ has positive  $(n-1)$-dimensional Hausdorff measure, and $y_0\in W^{1,\infty}(\O;\R^n)$ is a given function. Also recall that $\ell: W^{1,\infty}(\O;\R^n) \to \R$ is a linear bounded functional accounting for surface or volume forces. 
The last term of $J$ imitates
Dirichlet boundary conditions $y=y_0$ on $\Gamma$ realized by means of an elastic  hard device.  The constant $\alpha>0$ refers to the elastic properties of this device. From the mathematical point of view, this term, together with the locking constraint, will yield boundedness of the minimizing sequence of $J$ in $W^{1,\infty}(\O;\R^n)$ due to the generalized Poincar\'{e} inequality \cite[Theorem 1.32]{rou}. 
Notice also that the last term in $J$ is continuous with respect to the weak* convergence in  $W^{1,\infty}(\O;\R^n)$.

It is a classical result that if $W$ is quasiconvex then $J(y)$ is weakly* lower semicontinuous on $W^{1,\infty}(\O;\R^n)$ (see e.g.\ \cite{dacorogna}). Nevertheless, due to the involved locking constraint and the fact that $W$ may not even be defined outside $\overline{B(0,\varrho)}$, this does not directly mean that \eqref{functionalJ} possesses a solution. This issue was settled by Kinderlehrer and Pedregal \cite{k-p} who, by suitable rescaling, indeed showed that \eqref{functionalJ} is solvable provided $W$ is quasiconvex on its domain.

On the other hand, if $W$ is not quasiconvex, solutions to \eqref{functionalJ} might not exist due to faster and faster spatial  oscillations of the sequence of gradients. In this case, some physically relevant quantities \FIRST such as microstructure patterns \EOR  can be drawn from studying minimizers of the relaxed problem. In the following we provide a relaxation of the functional $J$ by means of Young measures. 

Thus, we define $\bar J: W^{1,\infty}(\O;\R^n)\times\GY\to\R$ and the following relaxed minimization problem via:
\begin{align}\label{functionalJbar} 
&\text{minimize} && \bar J(y_\nu,\nu):= \int_\O\int_{\R^{n\times n}} W(A)\nu_x(\md A)\,\md x-\ell(\SECOND y_\nu \EOR) + \alpha\|y_\nu-y_0\|_{L^2(\Gamma;\R^n)} ,\nonumber\\
&\text{subject to}&&  y_\nu\in W^{1,\infty}(\O;\R^n) ,\ \nu\in \GY,\  \text{supp}\,\nu\subset \overline{B(0,\varrho)},\
%\nonumber ,\\
%&\ && 
\nabla y_\nu=\bar\nu .\end{align}
\SECOND
Here we recall that $\bar{\nu}_x = \int_{\R{n \times n}} A \d \nu_x(A)$. \EOR

The next result shows that \eqref{functionalJbar} is indeed a relaxation of \eqref{functionalJ} in the sense specified in the proposition.

\begin{proposition}\label{prop:one}
\SECOND Let $\Omega$ be a bounded Lipschitz domain. \EOR \FIRST Let $\ell\in (W^{1,\infty}(\O;\R^n))^*$, \EOR $W\in C(\overline{B(0,\varrho)})$ for some $\varrho>0$, let $\alpha>0$, and let $y_0\in W^{1,\infty}(\Omega;\R^n)$. Then the infimum of $J$ in \eqref{functionalJ} is the same as the minimum of $\bar J$ in \eqref{functionalJbar}. Moreover, 
every minimizing sequence of \eqref{functionalJ} contains a subsequence which $Y$-converges  to a minimizer of \eqref{functionalJbar}. On the other hand, for every minimizer $(y_\nu, \nu)$ of \eqref{functionalJbar} there exists a minimizing sequence of \eqref{functionalJ} $\{y_k\}_{k \in \N}$ such that 
$$
J(y_k) \to \bar{J}(y_\nu, \nu).
$$
\end{proposition}

 \begin{remark}
The relaxation statement in Proposition \ref{prop:one} is different to other similar relaxation statements using Young measures (cf.\ e.g.~\cite{pedregal}). Indeed, in \cite{pedregal} and other works, the relaxation is obtained for a large class of energy densities \emph{at once}. However, here the relaxation is obtained only for \SECOND the functional \eqref{functionalJ} with one particular given energy density $W$. \EOR \FIRST More precisely, given a  pair $(y_\nu,\nu)$ which minimizes $\bar J$   we construct a minimizing sequence $\{y_k\}$ such that 
\eqref{young-measures} holds   for $Y_k:=\nabla y_k$, $\xi=1$ and $f:=W$ or real  multiples of $W$. Yet, this is completely enough to show the relaxation result in full generality.

\EOR

\SECOND
\begin{remark}
Let us remark that the construction of the recovery sequence for the minimizer in the proof of Proposition \ref{prop:one} can be taken in verbatim to construct a recovery sequence not only for the minimizer but for any  $(y_\nu, \nu) \in W^{1,\infty}(\Omega;\R^n) \times \mathcal{GY}^\infty(\Omega;\R^{n \times n})$ such that $\bar{J}(y_\nu, \nu)$ is of finite energy.
\end{remark}
\EOR

% On a more technical level, we cannot prove that for every minimizer of the relaxed problem \eqref{functionalJbar} there exists a minimizing sequence of the original problem \eqref{functionalJ} that generates it. We can only show existence of such a minimizing sequence along which the \emph{inspected functionals} converge to each other.
\end{remark}

%where $B(0,r):=\{A\in\R^{n\times n};\ |A|<r\}$  for some $r>0$.  We extend $v$ on the whole $\R^{n\times n}$ by $\infty$. 
\begin{proof}[Proof of Proposition~\ref{prop:one}]
We first prove that every minimizing sequence of \eqref{functionalJ} (or at least a subsequence thereof) $Y$-converges to a minimizer of \eqref{functionalJbar} and that the values of the infimum in \eqref{functionalJ} and the minimum of \eqref{functionalJbar} agree. 

To this end, take $\{y_k\}_{k\in\N}$ a minimizing sequence for \eqref{functionalJ}. This sequence has to belong to the  set \SECOND $\mathcal{A}_\varrho$ \EOR
and, by definition, $\{J(y_k)\}_{k\in\N}$ converges to $\inf_{\mathcal{A}_\varrho}J$.  Inevitably, $\nabla y_k(x)\in \overline{B(0,\varrho)}$ for almost all $x\in\O$, so that $L(\nabla y_k)\le 0$, with $L$ as in \eqref{locking}. Moreover, as $\alpha>0$ \SECOND and $\Omega$ is a bounded Lipschitz domain, the Poincar\'{e} inequality \cite[Thm.~1.32]{rou} \EOR implies that $\{y_k\}_{k \in \N}$ is uniformly  bounded in $W^{1,\infty}(\O;\R^n)$. Therefore, there is a Young measure $\nu\in\GY$ and a function $y_\nu\in W^{1,\infty}(\O;\R^n)$ such that $y_k\wY (y_\nu,\nu)$ as $k\to+\infty$ (at least in terms of a non-relabeled subsequence). Moreover, $\nu$ is supported in   %$\overline{\bigcap_k \mathrm{supp}(\nabla y_k)} \subset \overline{B(0,\varrho)}$ for almost all $x\in\O$  
$$\bigcup_{\mbox{\footnotesize a.e.\ } x\in \Omega} \mathrm{supp}(\nu_x) \subset \bigcup_{\footnotesize\mbox{a.e.\ } x\in \Omega} \bigcap_{\ell=1}^\infty \overline{\{ \nabla y_k(x), \, k\geq \ell\}} \subset \overline{B(0,\varrho)}$$
(see e.g.\ \SECOND \cite[Theorem I]{balder}\EOR) and for the first moment of $\nu$ we have that $\bar\nu=\nabla y_\nu$. Now, by the fundamental theorem on Young measures $J(y_k) \to \bar{J}(y_\nu, \nu)$ and, consequently, it holds that  $$\inf_{\mathcal{A}_\varrho}J= \bar J(y_\nu,\nu)\ .$$ 

We now prove that $(y_\nu,\nu)$ is indeed a minimizer of $\bar{J}$ in \eqref{functionalJbar}. Suppose, by contradiction, that this was not the case. Then there had to exist another gradient Young measure $\mu\in\GY$ and a corresponding $y_\mu\in W^{1,\infty}(\O;\R^n)$ such that $\nabla y_\mu=\bar\mu$, $\mu$ is supported on $\overline{B(0,\varrho)}$ and $$\bar J(y_\mu,\mu)<\bar J(y_\nu,\nu).$$
We show that this is not possible. 
Indeed,  for every $0<\varepsilon <1$ we find a generating sequence $\{y^\varepsilon_k\}_{k\in\N}\subset W^{1,\infty}(\O;\R^n)$ for $\mu$, that is $ y^\varepsilon_k\wY (y_\mu,\mu)$ as $k\to +\infty$, such that $\sup_{k\in\N}\|\nabla y^\varepsilon_k\|_{L^\infty(\O;\R^{n\times n})}\le \varrho+\varepsilon$, see  Proposition~\ref{mueller-proposition}. By the fundamental theorem on Young measures, 
\begin{align} \label{star}
\lim_{k\to+\infty} J(y^\varepsilon_k)=\bar J(y_\mu,\mu) < \bar J(y_\nu,\nu)=\inf_{\mathcal{A}_\varrho}J.
\end{align}
At this point, we abused the notation a bit for the sake of better readability of the proof. Indeed, $J(y^\varepsilon_k)$ might not be well-defined because $\nabla y^\varepsilon_k(x)$ might not be contained in the ball $\overline{B(0,\varrho)}$ for almost all $x \in \Omega$; however, the energy density $W$ is defined, originally, just on this ball. Nevertheless, relying on the Tietze theorem, we can extend $W$ from the ball $\overline{B(0,\varrho)}$ in a continuous way to the whole space. We will denote this extension by $W$ again and the functional into which it enters again by $J$. 

By \eqref{star}, there is $k_0=k_0(\varepsilon)\in\N$ such that $ J(y^\varepsilon_{k_0})\leq \inf_{\mathcal{A}_\varrho}J-\delta$ for some $\delta>0$ and $\|y^\varepsilon_{k_0}-y_\mu\|_{L^2(\Gamma;\R^n)}\le 1$ as well as $\|y^\varepsilon_{k_0}-y_\mu\|_{L^\infty(\Omega;\R^n)}\le 1$ (due to the strong convergence of $\{y^\varepsilon_k\}$ to $y_\mu$ for $k\to +\infty$).  Now, we can apply a trick, similar to the one used in \cite[Proposition~7.1]{k-p1}, and multiply $  y^\varepsilon_{k_0}$ by $\varrho/(\varrho+\varepsilon)$ so that the values of the gradient belong to  $\overline{B(0,\varrho)}$. With this rescaling, we obtain 
\begin{align*} 
\big |J(\varrho/(\varrho+\varepsilon)y^\varepsilon_{k_0})- J(y^\varepsilon_{k_0})\big| 
&  \le\left|\int_\O W(\varrho/(\varrho+\varepsilon)\nabla  y^\varepsilon_{k_0}(x))\md x - \int_\O W(\nabla y^\varepsilon_{k_0}(x))\md x\right|\nonumber\\
&\quad +C_\ell \left\|\varrho/(\varrho+\varepsilon)y^\varepsilon_{k_0}-y_{k_0}^\varepsilon\right\|_{L^\infty(\O;\R^n)}\nonumber\\
& \quad +\left| \alpha\left\|\varrho/(\varrho+\varepsilon)y^\varepsilon_{k_0}-y_0 \right \|_{L^2(\Gamma;\R^n)}- \alpha\left \|y^\varepsilon_{k_0}-y_0\right \|_{L^2(\Gamma;\R^n)}\right|\nonumber \\
&\le \int_\O \vartheta\big(\varepsilon/(\varrho+\varepsilon)|\nabla  y^\varepsilon_{k_0}(x)|\big)\,\md x+C_\ell \|\varrho/(\varrho+\varepsilon)y^\varepsilon_{k_0}-y_{k_0}^\varepsilon\|_{L^{\infty}(\O;\R^n)}\nonumber\\
& \quad + \alpha\| \varrho/(\varrho+\varepsilon)y^\varepsilon_{k_0}-y^\varepsilon_{k_0}\|_{L^2(\Gamma;\R^n)},\nonumber
\end{align*}
where $C_\ell$ is the norm of $\ell$ and  $\vartheta:[0,+\infty)\to[0,+\infty)$ is the nondecreasing  modulus of uniform continuity of $W$ on $\overline{B(0,\varrho+1)}$, which satisfies $\lim_{s\to 0}\vartheta(s)=0$. Since $\nabla y^\varepsilon_{k_0}$ is bounded by $\varrho + \varepsilon$, we have
\begin{align}
|J(\varrho/(\varrho+\varepsilon) y^\varepsilon_{k_0})- J(y^\varepsilon_{k_0})\big| 
& \le  \mathcal{L}^n(\O)\vartheta\big(\varepsilon\big)+\frac{(C_\ell+\alpha) \varepsilon}{\varrho+\varepsilon}(\|y^\varepsilon_{k_0}\|_{L^2(\Gamma;\R^n)}+\|y^\varepsilon_{k_0}\|_{L^\infty(\O;\R^n)})\nonumber\\
&\le \mathcal{L}^n(\O)\vartheta\big(\varepsilon\big)+\frac{(C_\ell+\alpha)\varepsilon}{\varrho+\varepsilon}(\|y_\mu\|_{L^2(\Gamma;\R^n)}+\|y_\mu\|_{L^\infty(\O;\R^n)}+2)  .\label{modulus}
  \end{align}

 %Namely, for every $\varepsilon>0$  there is $\{y^\varepsilon_k\}\subset W_0^{1,\infty}(\O;\R^n)$ bounded, such that $\{\nabla y^\varepsilon_k\}_{k\in N}$ generates $\mu$, and  $\sup_{k\in\N}\|\nabla y^\varepsilon_k\|_{L^2}\le r+\varepsilon$. 
%%We extend $v$ continuously to the whole $\R^{n\times n}$ by the Tietze's theorem. 
%We have $\lim_{k\to\infty} I(y^\varepsilon_k)=\bar I(\mu) < \bar I(\nu)=\inf_{\mathcal{A}_r}I$.  Hence, there is $k_0=k_0(\varepsilon)\in\N$ such that $ I(y^\varepsilon_{k_0})< \inf_{\mathcal{A}_r}I-\delta$ for some $\delta>0$.  Moreover, we can multiply  $\nabla  y_{k_0}$ by $r/(r+\varepsilon)$ so that its values belong to  $\overline{B(0,r)}$. We learned this trick from \cite[Proposition~7.1]{k-p1}. It yields 
%\begin{align}&\big |I(r(r+\varepsilon)^{-1}y^\varepsilon_{k_0})- I(y^\varepsilon_{k_0})\big|
%=\big|\int_\O v(r(r+\varepsilon)^{-1}\nabla \tilde y^\varepsilon_{k_0}(x))\md x - \int_\O v(\nabla y^\varepsilon_{k_0}(x))\md x\big|\nonumber \\
%&\le \int_\O \vartheta\big(\frac{\varepsilon}{r+\varepsilon}|\nabla  y^\varepsilon_{k_0}(x)|\big)\,\md x\ \le \mathcal{L}^n(\O)\vartheta\big(\varepsilon\big) .\label{modulus}
  %\end{align}
  The right-hand side in \eqref{modulus} tends to zero as $\varepsilon\to 0$; therefore,  for $\varepsilon>0$ small enough, it is smaller than  $\delta/2$ and thus  $$ J(\varrho/(\varrho+\varepsilon) y^\varepsilon_{k_0})\leq \inf_{\mathcal{A}_\varrho}J- \delta/2.$$  
 This closes our contradiction argument because $\varrho/(\varrho+\varepsilon) y^\varepsilon_{k_0}\in\mathcal{A}_\varrho$, so that we showed that every minimizing sequence of \eqref{functionalJ} generates a minimizer of \eqref{functionalJbar}.
 
To finish the proof, we need to show that for any minimizer $(y_\nu, \nu) \in W^{1,\infty}(\Omega;\R^n) \times \mathcal{GY}^\infty(\Omega;\R^{n \times n})$ of \eqref{functionalJbar} we can construct a sequence $\{y_k\}_{k \in \N} \subset W^{1,\infty}(\Omega; \R^n)$ that is a minimizing sequence of $\eqref{functionalJ}$ and satisfies
$$
J(y_k) \to \bar{J}(y_\nu, \nu).
$$

The strategy is similar to above: indeed, for any $\varepsilon>0$ we find a generating sequence $\{y^\varepsilon_k\}_{k\in\N}\subset W^{1,\infty}(\O;\R^n)$ for $\nu$, that is $ y^\varepsilon_k\wY (y_\nu,\nu)$ as $k\to +\infty$, such that $\sup_{k\in\N}\|\nabla y^\varepsilon_k\|_{L^\infty(\O;\R^{n\times n})}\le \varrho+\varepsilon$. Let us rescale this sequence by $\frac{\varrho}{\varrho+\varepsilon}$ and consider only $k$ large enough to obtain a sequence $\{\bar{y}^\varepsilon_k\}_{k\in\N}$ that is contained in $\mathcal{A}_\varrho$ and satisfies $\|\bar{y}^\varepsilon_{k}-y_\nu\|_{L^2(\Gamma;\R^n)}\le 1$ as well as $\|\bar{y}^\varepsilon_{k}-y_\nu\|_{L^\infty(\Omega;\R^n)}\le 1$ for all $k \in \N$. 

Choosing a subsequence of $k$'s if necessary, we can assure that 
$$
|J(y_k^{\varepsilon(k)}) - \bar{J}(y_\nu, \nu)| \leq \frac{1}{k},
$$
for any arbitrary $\varepsilon$ fixed. Moreover, owing to \eqref{modulus}, we can choose $\varepsilon=\varepsilon(k)$ in such a way that 
$$
\left|J(\bar{y}_k^{\varepsilon(k)})-J(y_k^{\varepsilon(k)})\right| \leq \frac{1}{k},
$$
so that
$$
\left|J(\bar{y}_k^{\varepsilon(k)})-\bar{J}(y_\nu, \nu)\right|\leq \frac{2}{k}.
$$
Thus,  we can construct a sequence $\{\bar{y}_k\}_{k\in\N}$ by setting $\bar{y}_k = \bar{y}_k^{\varepsilon(k)}$ that lies in $\mathcal{A}_\varrho$ and satisfies
$$
J(\bar{y}_k) \to \bar{J}(y_\nu,\nu) \quad \text{as } k\to +\infty.
$$
Finally, since $\bar{J}(y_\nu,\nu) = \inf_{\mathcal{A}_\varrho}J$ the constructed sequence is indeed a minimizing sequence of $J$.
\end{proof}

\begin{remark}
Notice that the previous result includes the case when $W(F)=+\infty$ whenever $L(F)>0$, i.e., when $|F|>\varrho$. 
\end{remark}

\begin{remark}
As mentioned in the introduction, Lipschitz continuous deformations naturally appear in the theory of non-simple materials, see e.g.~\cite{healey-kroemer,mielke-roubicek,podio}, where the stored energy density depends not only on the first gradient of the deformation but also on its higher gradients. In the simplest situation, one considers the first and the second gradient of $y$.  Then it is natural to  assume that $y\in W^{2,p}(\O;\R^n)$, which for $p>n$ embeds into $W^{1,\infty}(\O;\R^n)$  and makes $\nabla y$ H\"{o}lder continuous on $\overline{\O}$. 
\end{remark}

\SECOND
We may ask whether the recovery sequence constructed in the proof of Proposition \ref{prop:one} can be taken as piecewise-affine. This would allow to construct the recovery sequence, e.g., by finite element approximations or find its application in discrete-to-continuum transitions. The answer here is affirmative. Indeed,
\EOR
we quote the following proposition which can be found in \cite[Proposition~2.9, p.~318]{ekeland-temam}.
\begin{proposition} 
Let $\O\subset\R^n$ be a bounded Lipschitz domain, $\varrho>0$,  and  $y\in \mathcal{A}_\varrho$ as in \eqref{**}.
Then there is an increasing sequence of open sets,   $\O_{i-1}\subset\O_i\subset\O$, $i\in \N$, such that  $\mathcal{L}^n(\O\setminus\O_i)\to 0$ for $i\to+\infty$ and a sequence of Lipschitz maps  $\{\tilde y_i\}_{i\in\N}\subset W^{1,\infty}(\O;\R^n)$ such that $\tilde y_i$ is piecewise affine on $\O_i$, 
\begin{align*}\|\nabla \tilde y_i\|_{L^\infty(\O;\R^{n\times n})}\le \|\nabla y\|_{L^\infty(\O;\R^{n\times n})}+c_i ,\end{align*}
where $\lim_{i\to+\infty}c_i= 0$, $\tilde y_i\to y$ uniformly in $\O$, and $\nabla \tilde y_i\to\nabla y$  almost everywhere in $\O$ as $i\to+\infty$.  Moreover, $\tilde y_i=y$ on $\partial\O$ for all $i\in\N$.
\end{proposition}
% Hence, given $\varepsilon>0$ we can take $i\in\N$ large enough that 
%  \begin{align}
%  \big|\int_\O W(\nabla y(x))\,\md x-\int_{\O_i} W(\nabla y(x))\,\md x\big|<\varepsilon .
% \end{align}
% On the other hand, the Lebesgue dominated convergence theorem yields that if $i\ge i_0$  then 
%  \begin{align}
%  \big|\int_\O W(\nabla y(x))\,\md x-\int_{\O} W(\nabla \tilde y_i(x))\,\md x\big|<\varepsilon .
%  \end{align}
\SECOND With this proposition at hand, we may now sketch the construction of the piecewise-affine recovery sequence: \EOR Let us denote $\varrho_i:=\|\nabla \tilde y_i\|_{L^\infty(\O;\R^{n\times n})}$. As $\varrho_i\le\varrho+c_i$ we get  
 $$\varrho/(\varrho+c_i)\tilde y_i\in\mathcal{A}_\varrho .$$ 
 Moreover, $\varrho/(\varrho+c_i)\tilde y_i$ is also piecewise affine on $\O_i$. 
 Applying analogous reasoning as in \eqref{modulus} we get that 
 $|J(y)-J(\varrho/(\varrho+c_i)\tilde y_i)|$ is arbitrarily small if $i\in\N$ is large because $c_i\to 0$. Consequently, the infimum of $J$ can be approximated by maps which are piecewise affine  on open subsets of $\O$ the Lebesgue measure of which approaches $\mathcal{L}^n(\O)$. It follows from the proof of \cite[Proposition~2.9, p.~318]{ekeland-temam} that these subsets $\O_i$ are \SECOND polyhedral \EOR domains containing $\{x\in\O; \text{ dist}(x,\partial\O)>1/i\}$. \SECOND In fact, the only reason why one has to construct these subdomains $\Omega_i$ is that either the domain itself is not polyhedral or the boundary datum is not piecewise affine. \EOR  Hence, if $\O$ is already a \SECOND polyhedral \EOR  domain and $y$ is affine on each affine segment of $\partial\O$ then  for all $i\in\N$  we may set  $\O_i:=\O$ and $y_i$ can be taken piecewise affine on the whole $\O$ for all $i$.

 \bigskip

The provided relaxation in Proposition \ref{prop:one} utilizes (gradient) Young measures. While this is a useful tool often used in the calculus of variations, it is still more common in some applications to use a relaxation by means of the  so-called ``infimum-formula'' \eqref{inf-formula} below,  see e.g.~\cite[Proposition~7.2]{k-p1}.  We show in the next proposition how such an infimum formula can be phrased in terms of gradient Young measures. 
To this end, we
set for  $A \in \overline{B(0,\varrho)}$  
\begin{align*}
\mathcal{GY}^\infty_{A,\varrho}:= & \{ \nu \in \mathcal{GY}^\infty(\Omega; \R^{n \times n}): \text{$\nu$ is a homogeneous (i.e., independent of $x$) measure },\\
&\operatorname{supp} \nu \subset \overline{B(0,\varrho)} \text{ and } \bar{\nu} = A\}
\end{align*}
and 
\begin{align}\label{eq:relaxation}
W^\mathrm{rel}(A):=\min_{\nu \in \mathcal{GY}^\infty_{A,\varrho}} \int_{\R^{n \times n}} W(F) \nu(\d F).
\end{align}
\SECOND
Notice that a minimizer is guaranteed to exist in \eqref{eq:relaxation} by the direct method. Indeed, any minimizing sequence of \eqref{eq:relaxation} is automatically bounded on measures (since we work with probability measures here) and the first moment as well as the property of the support being contained in a $\varrho$-ball are preserved under weak$^\star$ convergence.
\EOR
\begin{proposition}
\label{prop:two}
Let $W\in C(\overline{B(0,\varrho)})$ for some $\varrho>0$. %Let the assumptions of Proposition \ref{prop:one} hold. 
For $A\in\R^{n\times n}$ set
\begin{equation}
\bar W(A)
:= \begin{cases}
 W^\mathrm{inf}(A) &\text{if $|A|<\varrho$,} \\
W(A) &\text{if $|A|=\varrho$},
\end{cases}
\label{inf-formula}
\end{equation}
where 
\begin{align*}
 W^\mathrm{inf}(A):= \inf \left\{ \tfrac{1}{\mathcal{L}^n(\O)} \int_{\Omega} W  (A + \nabla \phi) \d x : \phi\in  W_0^{1, \infty} (\Omega;\R^n) \text{ with } \FIRST  A\cdot+ \phi \in \mathcal{A}_\varrho \EOR \text{ a.e.\ in } \Omega \right\}.
 \end{align*}

Then  it holds that 
$$
W^\mathrm{rel}(A)=\bar W(A).
$$
\end{proposition}

\begin{proof}
In order to prove the claim, we will need the following homogenization result that is a slight variation of \cite[Theorem~2.1]{k-p} and is proved by a blow-up argument. 

\bigskip 

\begin{lemma}\label{homogenization}
Let  $\{u_k\}_{k\in\N}$ be a bounded sequence in $W^{1,\infty}(\O;\R^n)$ with $u_k(x)=Ax$ in $\partial \Omega$. 
Let the Young measure $\nu\in \mathcal{GY}^\infty(\Omega; \R^{n \times n})$ be generated by $\{\nabla u_k\}_{k\in\N}$. Then there is another bounded sequence $\{w_k\}_{k\in\N}\subset\FIRST  W^{1,\infty}(\Omega;\R^{n})$ \EOR with $w_k(x)=Ax$ in $\partial \Omega$ that generates a homogeneous measure  $\bar\nu$   defined through
\begin{align*}\label{homog} 
\int_{\R^{n\times n}} f(s)\bar\nu(\md s)= \frac{1}{\mathcal{L}^n(\O)}\int_{\O}\int_{\R^{n \times n}} f(s)\nu_x(\md s)\,\md x , \end{align*}
for any $f\in C(\R^{n \times n})$. %and almost all $x\in\O$. 
Moreover, if $\{\nabla u_k\}_{k\in\N} \subset \mathcal{A}_\varrho$ for a.e.\ $x \in \Omega$ then also $\mathrm{supp} \, \bar \nu \subset \overline{B(0,\varrho)}$.
\end{lemma}

Using this lemma, we prove  Proposition~\ref{prop:two} for $|A| < \varrho$. In this case, we can, for any $\phi \in W^{1,\infty}_0(\Omega;\R^n)$ with $A+\nabla \phi \in \mathcal{A}_\varrho$ define for almost every $x\in\O$ the Young measure $\nu_x:=\delta_{A+\nabla \phi(x)}$. According to the homogenization Lemma \ref{homogenization}, this defines a homogeneous Young measure \SECOND $\tilde{\mu} \EOR\in \mathcal{GY}_{A,\varrho}^\infty$ through
\SECOND 
$$
\int_{\R^{n\times n}} f(s)  \tilde \mu(\md s) = \frac{1}{\mathcal{L}^n(\O)}\int_{\O}\int_{\R^{n \times n}} f(s)\delta_{A+\nabla \phi(x)}(\md s)\,\md x , 
$$
\EOR
for any $f\in C(\R^{n \times n})$. \SECOND Plugging in here $f(s) = s$, we obtain that the first moment of $\tilde{\mu}$ is $A$ since $\phi$ is vanishing on the boundary. \EOR % and almost all $x\in\O$. 
Thus,
$$
\min_{\nu \in \mathcal{GY}^\infty_{A,\varrho}} \int_{\R^{n \times n}} W(F) \nu(\d F) \leq \int_{\R^{n \times n}} W(F) \bar{\mu}(\d F) = \frac{1}{\mathcal{L}^n(\O)} \int_{\Omega} W  (A + \nabla \phi) \d x;
$$
and taking the infimum on the right hand side gives that $$\min_{\nu \in \mathcal{GY}^\infty_{A,\varrho}} \int_{\R^{n \times n}} W(F) \nu(\d F) \leq \bar W(A)$$
by \eqref{inf-formula}.

 On the other hand, let us take a sequence $\{\phi_k\}_{k \in \N} \subset W^{1,\infty}_0(\Omega;\R^n)$ such that 
$\frac{1}{\mathcal{L}^n(\O)} \int_{\Omega} W  (A + \nabla \phi_k) \d x \to \bar W(A)$. Then $\{A+\nabla \phi_k\}_{k \in \N}$ generates a Young measure $\nu_x \in \mathcal{GY}^\infty(\Omega;\R^{n \times n})$ so that 
$$
\bar W(A) = \lim_{k \to +\infty} \frac{1}{\mathcal{L}^n(\O)} \int_{\Omega} W  (A + \nabla \phi_k) \d x = \frac{1}{\mathcal{L}^n(\O)} \int_{\Omega} \int_{\R^{n\times n}} W  (s) \nu_x (\d s) \d x = \int_{\R^{n\times n}} W(s) \SECOND \tilde{\nu} (\d s) \EOR,
$$
where the homogeneous measure \SECOND $\bar{\nu} \in \mathcal{GY}^\infty_{A, \varrho}$ \EOR is defined according to Lemma \ref{homogenization}. \SECOND Notice again that taking $W(F) = F$ in the above formula guarantees the right first moment on $\tilde{\nu}$. \EOR Therefore, we obtain that 
$$
\bar W (A) \geq \min_{\nu \in \mathcal{GY}^\infty_{A,\varrho}} \int_{\R^{n \times n}} W(F) \nu(\d F).
$$

Let us now concentrate on the case when $|A|=\varrho$. In this case, we use that $\overline{B(0,\varrho)}$ is a strictly convex set (i.e., if $|A_1|=|A_2|=\varrho$, $A_1\neq A_2$, then $|\lambda A_1+(1-\lambda)A_2|<\rho$ for all $0<\lambda<1$) so that the only homogeneous Young measure that is supported in $\overline{B(0,\varrho)}$ and satisfies \SECOND that the first moment of  $\tilde{\nu}$ denoted $\overline{\tilde \nu}$ equals $A$  is the Dirac measure $\delta_A$. Indeed, the modulus of the first moment is a convex function of the measure supported on  $\overline{B(0,\varrho)}$. Hence it is maximized at some extreme point, i.e., if $\tilde\nu=\delta_A$\    cf.~also \cite[Theorem~8.4 on p.~147]{conway}. \EOR From this, it readily follows that 
$$
\min_{\nu \in \mathcal{GY}^\infty_{A,\varrho}} \int_{\R^{n \times n}} W(F) \nu(\d F) = \int_{\R^{n \times n}} W(F) \delta_A(\d F) = W(A) \quad \text{if} \quad |A|=\varrho.
$$
\end{proof}

Let us point out that the infimum formula \eqref{inf-formula} obtained in Proposition \ref{prop:two} strongly relies on the fact that the locking constraint based on \eqref{locking} requires the strains to be constrained to a \emph{strictly convex} region. Of course, one could imagine more general situations, in which the locking constraint requires the strains to lie in $\overline{\mathcal{R}_\varrho}$, where $\mathcal{R}_\varrho$ is a convex open set containing $0$ that, however, is not necessarily  strictly convex. In such a situation the infimum formula is more involved and has been found by Wagner \cite{wagner}. We show in the next proposition that it follows easily from the Young measure representation.

\begin{proposition} \label{rem34}
Let $\mathcal{R}_\varrho$ be a convex open set containing $0$ and let $W: \R^{n\times n} \to [0,+\infty]$ be defined as
$$ W(F) \begin{cases} <+\infty & \text{if } F \in \overline{\mathcal{R}_\varrho}, \\ 
= +\infty & \text{if } F \notin \overline{\mathcal{R}_\varrho}\end{cases}$$
and assume that $W$ is continuous on its domain. Then the relaxation of $W$  reads
\begin{equation}
W^\mathrm{rel}(A)
= \begin{cases}
W^\mathrm{inf}(A) &\text{if $A \in \mathcal{R}_\varrho$,}  \\
\lim_{\varepsilon \to 0} W^\mathrm{inf}\left(\frac{|A|-\varepsilon}{|A|} A\right)  &\text{if $A \in \partial \mathcal{R}_\varrho$},
\end{cases}
\label{infFormula2}
\end{equation}
where 
$$
W^\mathrm{inf}(A) := \inf \Big{ \{\frac{1}{\mathcal{L}^n(\O)}} \int_{\Omega} W  (A + \nabla \phi) \d x : \phi \in W_0^{1, \infty} (\Omega) 
\text{ with } A+\nabla \phi \in \overline{\mathcal{R}_\varrho} \text{ a.e.\ in $\Omega$} \Big \}. 
$$
\end{proposition}
Thus, the characterization in the domain $\mathcal{R}_\varrho$ stays the same as above but at the boundary $\partial \mathcal{R}_\varrho$ we replace the original function by a radial limit of the relaxation obtained inside the domain. Actually, Wagner \cite{wagner-rel} provides an example showing that $W^\mathrm{rel}$ obtained by \eqref{inf-formula} and \eqref{infFormula2}, respectively, differ for not strictly convex domains. 

%Here we show that also the relaxation formula \eqref{infFormula2} can be obtained from the Young measure representation. 

\begin{proof}
To show \eqref{infFormula2}, we have just to prove the representation on the boundary $\partial \mathcal{R}_\varrho$, since inside the domain the proof is the same as the one given for Proposition \ref{prop:two}.

To do so, let us first realize that having a sequence of Young measures $\{\mu_k\}_{k \in \N}$ supported in $\overline{\mathcal{R}_\varrho}$ that converges weakly$^\star$ to another Young measure $\mu$, we can assure that $\mu$ is supported in $\overline{\mathcal{R}_\varrho}$, too. Indeed, this claim can be easily seen by using test functions that are one on any Borel set in $\R^{n \times n} \setminus \overline{\mathcal{R}_\varrho}$ and zero in $\overline{\mathcal{R}_\varrho}$. By using this claim, we see by the direct method that the minimum  
$$
\min_{\nu \in \mathcal{GY}^\infty_{A,\mathcal{R}_\varrho}} \int_{\R^{n \times n}} W(F) \nu(\d F) , 
$$
where
\begin{align*}
\mathcal{GY}^\infty_{A,\mathcal{R}_\varrho}& = \{ \nu \in \mathcal{GY}^\infty(\Omega; \R^{n \times n}): \text{$\nu$ is a homogeneous measure with } \mathrm{supp} \, \nu \subset \overline{\mathcal{R}_\varrho} \text{ and } \bar{\nu} = A\} 
\end{align*}
is attained for all $A \in \overline{\mathcal{R}_\varrho}$. Let us thus select the minimizer for some given $A \in \partial \mathcal{R}_\varrho$, called $\nu_A$. Further, define the homogeneous measure $\tilde\nu^\varepsilon\in \mathcal{GY}^\infty_{\frac{|A|-\varepsilon}{|A|} A, \mathcal{R}_\varrho}$ for every $f\in C(\R^{n\times n})$ 
by the formula 
$$
\int_{\R^{n\times n}}f(F)\tilde\nu^\varepsilon(\md F)=\int_{\R^{n\times n}}f\left(\frac{|A|-\varepsilon}{|A|} F\right)\nu_A(\md F)\ .
$$
Therefore, we have that 
\begin{align*}
 \int_{\R^{n \times n}} W\left(F\right) \tilde\nu^\varepsilon(\d F)&=\int_{\R^{n \times n}} W\left(\frac{|A|-\varepsilon}{|A|}F \right) \nu_A(\d F) 
\\
&\geq \min_{\nu \in \mathcal{GY}^\infty_{\frac{|A|-\varepsilon}{|A|} A, \mathcal{R}_\varrho}} \int_{\R^{n \times n}} W(F) \nu(\d F) = W^\mathrm{inf} \left(\frac{|A|-\varepsilon}{|A|} A \right)
\end{align*}
since $\frac{|A|-\varepsilon}{|A|} A$ is in the interior of $\mathcal{R}_\varrho$ and thus Proposition~\ref{prop:two} applies.
Taking the limit $\varepsilon \to 0$, we get, relying on the continuity of $W$ on its domain,
\SECOND
\begin{align*}
\lim_{\varepsilon \to 0} W^\mathrm{inf} \left(\frac{|A|-\varepsilon}{|A|} A \right) & \leq \lim_{\varepsilon \to 0} \int_{\R^{n \times n}} W\left(\frac{|A|-\varepsilon}{|A|}F \right) \nu_A(\d F) \\   &=\int_{\R^{n \times n}} W(F) \nu_A(\d F) =\min_{\nu \in \mathcal{GY}^\infty_{A,\mathcal{R}_\varrho}} \int_{\R^{n \times n}} W(F) \nu(\d F). \ 
\end{align*}
\EOR
On the other hand, let us define a sequence of homogeneous Young measures $\mu_\varepsilon$ defined through
$$
\mu_\varepsilon =  \mathrm{argmin}_{\nu \in \mathcal{GY}^\infty_{\frac{|A|-\varepsilon}{|A|} A, \mathcal{R}_\varrho}} \int_{\R^{n \times n}} W(F) \nu(\d F).
$$
At least a subsequence (not relabeled) of this sequence converges weakly$^\star$ to another homogeneous Young measure $\mu \in \mathcal{GY}^\infty_{A,\mathcal{R}_\varrho}$. Therefore, we have that
$$
\lim_{\varepsilon \to 0}  \int_{\R^{n \times n}} W(F) \mu_\varepsilon(\d F) \geq   W^\mathrm{rel}(A) .
$$
Altogether, we see that $ \lim_{\varepsilon \to 0} W^\mathrm{inf} \left(\frac{|A|-\varepsilon}{|A|} A \right) = W^\mathrm{rel}(A)$ at least for the subsequence selected above. Nevertheless, since the limit is uniquely determined, we can deduce that convergence is obtained even along the whole sequence.
\end{proof} 

\section{Existence of minimizers under locking constraint on the determinant }\label{lock-det}

In this section, we consider the locking constraint \SECOND $L\leq 0$ with $L$ as in \EOR \eqref{locking-det}, i.e., we will work with deformation gradients that can only lie in the set 
$$
\mathcal{S}_\varepsilon=\{F\in \R^{n \times n}: \det F \geq \varepsilon\}
$$
for some $\varepsilon > 0$.
Before embarking into our discussion, let us stress that imposing the above constraint puts us, from the mathematical point of view, into a very different situation than the constraint based on \eqref{locking}. Indeed, the set $\mathcal{S}_\varepsilon$ is non-convex while the strains constrained by \eqref{locking} lie in a convex set. 

Let us note that relaxation on the set $\mathcal{S}_\varepsilon$ (or even the case $\det F > 0$) is largely open to date and only scattered results can be found in the literature; cf.~e.g.~\cite{benesova-kamp,benesova-kruzik,krw}. In fact, even the existence of minimizers for quasiconvex energies that take the value $+\infty$ outside the set $\{F\in \R^{n \times n}: \det F > 0\}$ remains open to-date (see \cite{ball-puzzles,benesova-kruzik-wlsc}). %One can view the proof of existence of minimizers for energy densities infinite outside the ``locked set'' as a prerequisite for a relaxation result, which also explains why we will not show a relaxation result here.  

We do not tackle those issues here and study energy densities that are \emph{quasiconvex and finite on the whole space $\R^{n \times n}$ (and not just on the set $\mathcal{S}_\varepsilon$, differently to Section \ref{sec-relaxation}) but constrain the deformation gradients to lie in $\mathcal{S}_\varepsilon$}.
%but will prove existence of minimizers for some energies with the deformation gradients lying in .

%We assume here that the . We start by summarizing the available results on the existence of minimizers under the locking constraint \eqref{locking-det}. 
In more detail, let $\O\subset\R^n$ with $n=2,3$ and let $W$ be a continuous function on $\R^{n \times n}$ that is  quasiconvex. Then, we consider the following minimization problem:
\begin{align}\label{functionalJ2} 
&\text{minimize} && J(y):=\int_\O W(\nabla y(x))\,\md x-\ell(y)  ,\nonumber\\
&\text{subject to}&& \nabla y(x)\in\mathcal{S}_\varepsilon \cap \overline{B(0,\varrho)} \text{ a.e.\ in $\Omega$} ,\  y\in W^{1,\infty}(\O;\R^n) ,\ y=y_0\ \md A \mbox{ a.e.\ on }\ \Gamma,
\end{align}
where $y_0$ is a given function in $W^{1,\infty}(\O;\R^n)$ the gradient of which lies in $\mathcal{S}_\varepsilon$ almost everywhere. Moreover, $\Gamma\subset\partial\O$ has a positive  $(n-1)$-dimensional Hausdorff measure and $\ell$ is a bounded linear functional on $W^{1,\infty}(\Omega; \R^n)$.

 For simplicity, we work on Lipschitz functions only, although a generalization to Sobolev deformations lying in $W^{1,p}(\Omega;\R^n)$ with $p > n$ is possible with minor changes in the proofs.
\begin{remark}
The boundary datum $y_0$ in problem \eqref{functionalJ2} is defined on $\Omega$, while it would be more natural to define it on $\Gamma$  only.  Nevertheless, it is an open problem to-date to explicitly characterize the class of boundary data that allow for an extension as needed in problem \eqref{functionalJ2}, cf.\ the proof below. We refer to \cite[Section 7]{benesova-kruzik-wlsc} for a related discussion.
\end{remark}

We then have the following:

\begin{proposition}\label{prop:two1}
Let $\varepsilon >0$, $\varrho >\sqrt{n}\varepsilon^{1/n}$, $\ell\in (W^{1,\infty}(\O;\R^n))^*$. Let $W:\R^{n\times n}\to\R$ be quasiconvex, continuous, and let  $y_0\in W^{1,\infty}(\O;\R^n)$, $\nabla y_0\in\mathcal{S}_\varepsilon \SECOND \cap \overline{B(0,\varrho)} \EOR$ a.e.\ in $\Omega$. Then there exists a solution to \eqref{functionalJ2}.
\end{proposition}

\begin{proof}
The set of deformations admissible for \eqref{functionalJ2} is nonempty. Indeed, $\mathcal{S}_\varepsilon\cap\overline{B(0,\varrho)}$ contains the diagonal matrix with nonzero entries equal to $\varepsilon^{1/n}$.   Thus, we can select $\{y_k\}_{k\in\N}\subset W^{1,\infty}(\O;\R^n)$ a minimizing sequence of $J$ admissible in \eqref{functionalJ2}, i.e., $J(y_k)\to\inf\, J$ for $k\to+\infty$. Since $\nabla y_k$ is constrained to the ball $\overline{B(0,\varrho)}$,  there is $C>0$ such that $\sup_{k\in\N}\|y_k\|_{W^{1,\infty}(\O;\R^n)}\le C$ by the Poincar\'{e} inequality. Of course, we also have that  $\det\nabla y_k\ge \varepsilon>0$ in $\O$. Therefore, there is a (non-relabeled) subsequence such that 
$y_k\wstar y$ in $W^{1,\infty}(\O;\R^n)$ as $k\to+\infty$ for some $y\in W^{1,\infty}(\O;\R^n)$. Moreover, as $\det\nabla y_k\wstar\det\nabla y$ in $L^\infty(\O)$ for $k\to+\infty$ (see e.g.~\cite[Theorem~8.20]{dacorogna}), we have that $\det\nabla y\ge  \varepsilon$; thus, $\nabla y(x) \in \mathcal{S}_\varepsilon \cap \overline{B(0,\varrho)}$ a.e.\ in $\Omega$. By the standard trace theorem, $y=y_0$ on $\Gamma$. Summing up, we see that $y$ is an admissible deformation in \eqref{functionalJ2}. Quasiconvexity of $W$ and linearity of $\ell$ imply that 
$J$ is  lower semicontinuous along weakly* converging sequences of Lipschitz maps. Therefore, 
$J(y)\le\liminf_{k\to+\infty}J(y_k)$  and, consequently, $y$ is a solution.
\end{proof}

\section{Gradient polyconvexity}\label{gradient-polyconvexity}
\SECOND

In Section \ref{lock-det} we examined the locking constraint $\det F \geq \varepsilon$. Recall that in order to prove the results obtained there we needed the energy density to be finite and quasiconvex on the whole space; however, in many physical applications this might not be an admissible option. Indeed, as already mentioned in the introduction, physical energy densities should blow-up as the Jacobian of the deformation approaches zero. One possibility to incorporate this restriction is to let the energy depend (on parts) of the second gradient of the deformation. The canonical way to do so is to let the energy density be a \emph{convex} function of the second gradient (cf., e.g., \cite{ball-crooks,ball-mora,mielke-roubicek,podio,Silh88PTNB}). 

Yet, here we propose a different approach inspired by the notion of polyconvexity due to Ball \cite{ball77}. Indeed, in three dimensions, we consider energies that depend on the \emph{gradient of the cofactor and the gradient of the determinant}; i.e.
$$
I(y) = \int_\Omega \hat W(\nabla y(x), \nabla[ \cof \nabla y(x)], \nabla[ \det \nabla y(x)]) \d x
$$
for any deformation $y$. We shall call such energy functionals \emph{gradient polyconvex} if the dependence of $\hat{W}$ on the last two variables is convex. In this case, assuming also suitable coercivity of the energy, we prove not only existence of minimizers to the functional $I$ but also that it automatically satisfies the constraint $\det F \geq \varepsilon$; so, in other words, it can be used beyond the limitations on the energy density from Section \ref{lock-det}.
Let us also note that, as we show in Example \ref{exa-grad-poly}, the deformation entering such an energy needs not to be a $W^{2,1}(\Omega; \R^3)$-function, i.e., needs not to have an integrable second gradient. 

\WE 
Let us remark at this point that the notion of gradient polyconvexity works so well since the cofactor and the determinant have a prominent position not just from the point of view of weak continuity (they are null-Lagrangians) but also from the physical point of view. Indeed, the cofactor describes the deformation of surfaces while the determinant describes the deformation of volumes. \MAY Here, and everywhere in this section, we will limit our scope to $n=3$ for better readability but  the results hold in every   dimension. \EOR We start the detailed discussion with a definition of gradient polyconvexity.

\FIRST

\begin{definition} \label{def-gpc}
Let \MAY Let $\O\subset\R^3$ be a bounded open domain. \EOR  Let $\hat{W}:\R^{3\times 3}\times\R^{3\times 3\times 3}\times\R^3\to\R\cup\{+\infty\}$ be a lower semicontinuous function. The functional
\begin{align}\label{full-I}
I(y) = \int_\Omega \hat W(\nabla y(x), \nabla[ \cof \nabla y(x)], \nabla[ \det \nabla y(x)]) \d x,
\end{align}
defined for any measurable function $y: \Omega \to \R^3$ for which the weak derivatives $\nabla y$, $\nabla[ \cof \nabla y]$, $\nabla[ \det \nabla y]$ exist and are integrable is called {\em gradient polyconvex} if the function $\hat{W}(F,\cdot,\cdot)$ is convex for every \MAY $F\in\R^{3\times 3}$. \EOR
\end{definition}

%We now  turn our attention to the already announced {\it gradient polyconvexity}. Thus, as advertised in the introduction, we will consider a special type of non-simple materials that do not depend on the second gradient but depend, in a convex way, on the the gradient of the cofactor and possibly also on the  gradient of the determinant. We shall see in Example \ref{exa-grad-poly} that for such materials the deformation is not necessarily a $W^{2,1}(\Omega; \R^3)$-function. Nonetheless, we will be able to prove existence of stable states in Propositions~\ref{prop-grad-poly} and \ref{prop-grad-poly3} without any restriction on the growth (from above) on the energy density.

\begin{remark}
\label{rem-coercivity}
Let us note that Definition~\ref{def-gpc} includes the case in which $\hat{W}$ does not depend on some of its variables at all; in principle, a continuous $\hat{W}$ not depending on the gradients of the cofactor and the determinant  in any form would also meet the requirements of the definition. Nonetheless, we would not be able to prove existence of minimizers for the  corresponding  energy  functional $I$ since it would not satisfy suitable \emph{coercivity conditions}. Indeed, weak lower semicontinuity of integral functionals  defined in \eqref{energy-funct} 
is a crucial ingredient of proofs showing existence of minimizers. It relies on (generalized) convexity of the integrand  but  also on its  coercivity  conditions. Indeed, we refer to \cite[Example 3.6]{benesova-kruzik}  for a few examples of integral functionals as in \eqref{energy-funct} where $W$ is polyconvex but noncoercive and consequently $I$ is not weakly lower semicontinuous.
Therefore, we must also assume suitable  growth conditions for the energy density of the gradient polyconvex functional to be able to find minimizers by the direct method (see \cite{dacorogna}), see \eqref{growth-graddet1}.
\end{remark}

\EOR

% \begin{align}\label{newfunctional}
% J(y)&:=\int_\O W(\nabla y(x))\,\md x +\alpha\big(\|\cof\nabla y\|^q_{W^{1,q}(\O;\R^{3\times 3})}+\|\det\nabla y\|^r_{W^{1,r}(\O)}\big)\nonumber \\ 
% &\qquad +\beta\|(\det\nabla y)^{-s}\|_{L^1(\O)} 
% -\ell(y) ,
% \end{align}
% where $\alpha, \beta$ are positive constants, $q, r \geq 1$, $s>0$, and $\ell(y)$ is a linear functional on $W^{1,p}(\Omega;\R^3)$. 

By Definition~\ref{def-gpc}, the energy densities of gradient polyconvex functionals depend on the gradients of the determinant and of the cofactor. However, recall from \eqref{second-order} that, in the most standard setting in non-simple materials, the overall energy rather reads as $\int_\Omega w(\nabla y) + \gamma |\nabla^2 y|^d \mathrm{d} x$ for some $d \in [1, +\infty)$, $\gamma>0$, and a continuous function $w:\R^{3\times 3}\to [0,+\infty)$ representing the stored energy density. In other words, in the standard setting we can expect a deformation of finite energy to be contained at least in $W^{2,d}(\Omega;\R^3)$ while in our case we can only expect $y \in W^{1,p}(\Omega;\R^3)$, $\cof\nabla y \in W^{1,q}(\Omega; \R^{3 \times 3})$ and  $\det \nabla y \in W^{1,r}(\Omega)$.  Let us first realize that the former regularity implies the latter one with a proper choice of $d,p,r$, and $q$. Indeed, if $y\in W^{2,d}(\O;\R^3)$ then we have for $i,j,k,l,m\in\{1,2,3\}$  (Einstein's summation convention applies) by Cramer's rule \MAY $(\det F)\mathrm{Id} = (\cof F) F^T$ (with  ``$\mathrm{Id}$''  the identity matrix) \EOR  that%\eqref{cofac} 
\begin{align}
\frac{\partial}{\partial x_i}\det\nabla y=(\cof\nabla y)_{jk}\frac{\partial^2 y_j}{\partial x_k\partial x_i}\  \quad \text{ and } \quad \frac{\partial}{\partial x_i}(\cof\nabla y)_{jk}=\mathcal{L}_{jklm}(\nabla y)\frac{\partial^2 y_l}{\partial x_m\partial x_i}\ 
\end{align}
where 
$\mathcal{L}_{jklm}(F):=\frac{\partial(\cof F)_{jk}}{\partial F_{lm}} $
is a linear (or zero) function in $F$.
Hence, we see that gradients of nonlinear minors are controlled by the first and the second gradient of the deformation.  

On the contrary, the other implication does not hold as the following example shows: 
\begin{example}
\label{exa-grad-poly}
Let us note that, requiring for a deformation $y:\Omega \to \R^3$ to satisfy $\det \nabla y \in W^{1,r}(\Omega)$ and $\cof\nabla y \in W^{1,q}(\Omega; \R^{3 \times 3})$ is a weaker requirement than $y \in W^{2,1}(\Omega; \R^{3})$ for any $r,q \geq 1$. To see this, let us take $\Omega = (0,1)^3$ and the following deformation for some  $t\ge 1$
$$
y(x_1,x_2,x_3):= \left(x_1^2, x_2\, x_1^{t/(t+1)}, x_3\,  x_1^2 \right) , $$
$$
\text{ so that } \nabla y(x_1,x_2,x_3) = \left( {\begin{array}{ccc}
 2 x_1 & 0 & 0 \\ 
 \frac{t}{t+1}x_2\,x_1^{-1/(t+1)} & x_1^{t/(t+1)} & 0 \\ 
 2\,x_1 \, x_3 & 0 & x_1^2
  \end{array} } \right).
$$
It follows that 
$$
\det \nabla y(x_1,x_2,x_3) = 2x_1^{(4t+3)/(t+1)}>0$$ and 
$$\cof \nabla y(x_1,x_2,x_3) = \left( {\begin{array}{ccc}
  x_1^{(3t+2)/(t+1)}& -\frac{t}{t+1}x_2\,x_1^{(2t+1)/(t+1)} & -2\, x_1^{(2t+1)/(t+1)}\, x_3  \\ 
0 & 2\, x_1^3 & 0 \\ 
 0& 0 & 2\,x_1^{(2t+1)/(t+1)}
  \end{array} } \right) .
$$
Notice that $\det\nabla y\in W^{1,\infty}(\O)$, $\cof\nabla y\in W^{1,\infty}(\O;\R^{3\times 3})$, $ (\det\nabla y)^{-1/(4t+3)}\in L^1(\O)$ but  we see that $\nabla^2 y\not\in L^1(\O;\R^{3\times 3\times 3})$ which means that $y\not\in W^{2,1}(\O;\R^3)$. On the other hand, $y\in W^{1,p}(\O;\R^3)\cap L^\infty(\O;\R^3)$ for every $1\le p<1+t$.

\begin{figure}[h]
\center
\begin{minipage}{0.45\textwidth}
\center
\hspace*{-2.8cm}\includegraphics[height=0.4\textheight]{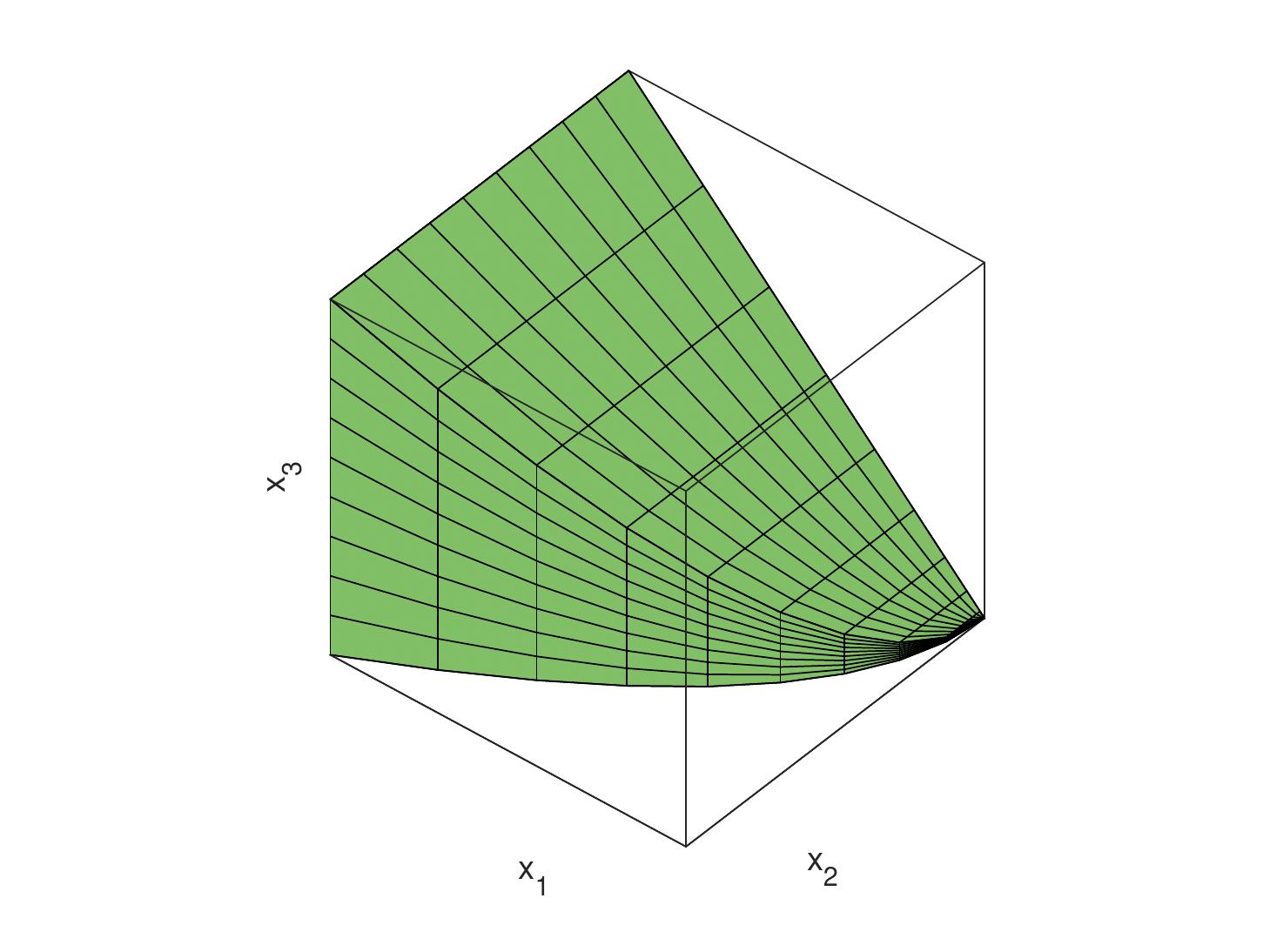}
\end{minipage}
\hspace*{-1.5cm}
\caption{Deformed cube (green)  in the frame of the reference domain  $(0,1)^3$ as in Example~\ref{exa-grad-poly} for $t=100$.}
\label{figure1}
\end{figure}

%Thus, it is not equivalent to replace $\big(\|\cof\nabla y\|^q_{W^{1,q}(\O;\R^{3\times 3})}+\|\det\nabla y\|^r_{W^{1,r}(\O)}\big)$ by a Sobolev norm of $\nabla y$ or by $\int_\Omega |\nabla^2 y|^\beta$ for some suitable $\beta > 1$, which is the most typical situation found in non-simple %second grade 
%materials. With a proper choice of $\beta, r, q$, the required regularity of the deformation to make the energy in \eqref{newfunctional} finite is indeed \emph{weaker} than the standard one in second grade materials. We refer to \cite{ball-crooks,ball-mora} for detailed discussions on various forms of the higher-order-gradient energy which is usually called interfacial in models of phase transition.
\end{example}

\WE
Therefore, the setting of gradient polyconvex materials is indeed more general \MAY than \EOR the standard approach used in non-simple materials involving the second gradient of the deformation. 
 \EOR

\begin{remark}
(i) If $n=2$ and $F\in\R^{2\times 2}$ then $\cof F$ has the same set of  entries as $F$ (up to the minus sign at off-diagonal entries), so that $J$ in \eqref{newfunctional} in fact depends on $\nabla^2 y$.\\
(ii) If $y:\O\to\R^3$ is smooth and  if $\nabla[\cof\nabla y]=0$,  almost everywhere in $\O \subset \R^3$ then for almost all $x\in\O$ we have $y(x)=Ax+b$ for some $A\in\R^{3\times 3}$ and $b\in\R^3$, and vice versa. \\
(iii) We recall that  $\det\nabla y$ and $\cof\nabla y$ measure   volume and area changes, respectively, between the reference  and the deformed configurations. Therefore, gradient polyconvexity ensures that these changes are not too  ``abrupt''.
\end{remark}

 The principle of frame indifference requires that $\hat W(\nabla y(x), \nabla [\cof \nabla y(x) ], \nabla[\det \nabla y(x)])$ equals $ \hat W(R\nabla y(x), \nabla [\cof R\nabla y(x)], \nabla[\det R\nabla y(x)])$ for every $y:\O\to\R^3$ as in Definition~\ref{def-gpc} and every $R\in SO(3)$. Since  $\nabla [\det R\nabla y]=\nabla [\det\nabla y]$ and $\nabla [\cof (R\nabla y)] = \nabla [\cof R\cof \nabla y] = \nabla [R \cof \nabla y] = R\nabla [\cof \nabla y]$, frame indifference translates to 
\begin{align}\label{pfi-2}
\hat W(F, \Delta_1,\Delta_2)=\hat W(RF, R\Delta_1,\Delta_2)
\end{align}
for every $F\in\R^{3\times 3}$, every $\Delta_1\in\R^{3\times3\times 3}$, every $\Delta_2\in\R^{3}$, and every proper rotation  $R\in{\rm SO}(3)$.
We recall that componentwise  $[R\Delta_1]_{ijk}:= \sum_{m=1}^3R_{im}[\Delta_1]_{mjk}$ for all $i,j,k\in\{1,2,3\}$.\\

\SECOND

Let us now turn to our existence theorems for gradient polyconvex energies. As already pointed out in Remark \ref{rem-coercivity}, we can do so, only when prescribing suitable coercivity conditions. We are going to assume, essentially, two types of growth conditions: In the first case, $\hat{W}$ does not actually depend on $\nabla[\det\nabla y]$, so that we assume that for some $c>0$, and finite numbers  $p,q,r,s\ge 1$ it holds that  
\begin{align}\label{growth-graddet1}
\hat W(F,\Delta_1)\ge\begin{cases}
c\big(|F|^p  +|\cof F|^q+(\det F)^r+ (\det F)^{-s}+|\Delta_1|^q \big)
&\text{ if }\det F>0,\\
+\infty&\text{ otherwise.} \end{cases}
\end{align}

Notice that even if the energy does not depend on $\nabla[\det\nabla y]$, we will be able to prove  not only existence of minimizers but also H\"older continuity  of the Jacobian. This is due to the fact that for every invertible  $F\in\R^{n\times n}$ 
\begin{align}\label{cofac}
\cof F:= (\det F) F^{-\top}\in\R^{n\times n} \ .\end{align}
and consequently 
\begin{align}\label{detcof-n}
\det\cof F= (\det F)^{n-1}\;
\end{align}
so that H\"older continuity of the cofactor also yields (local) H\"older continuity for the determinant. Notice also that if $F\in\R^{3\times 3}$ with $\det F>0$ then $F^{-\top}=\cof F/\sqrt{\det\cof F}$, so that $\cof F$ fully characterizes $F$.
\SECOND 
 Therefore, Proposition \ref{prop-grad-poly} is formulated in such a way that a locking constraint $L(\nabla y)\le 0$ is included. Setting $L:=0$ makes this condition void.

In the second case, we let $\hat{W}$ depend additionally on the gradient of the Jacobian and assume that  for some $c>0$, and finite numbers  $p,q,r,s\ge 1$ it holds that  
\begin{align}\label{growth-graddet2}
\hat W(F,\Delta_1,\Delta_2)\ge\begin{cases}
c\big(|F|^p  +|\cof F|^q+(\det F)^r+ (\det F)^{-s}+|\Delta_1|^q + |\Delta_2|^r\big)
&\text{ if }\det F>0,\\
+\infty&\text{ otherwise.} \end{cases}
\end{align}
This will allow us to broaden the parameter regime for $q$ in Proposition \ref{prop-grad-poly2} and still be able to prove existence of minimizers along with the locking constraint $\det F \geq \varepsilon$.

\EOR
In the following existence theorems, we consider gradient polyconvex functionals that allow also for a linear perturbation. Assume that $\ell$ is a  continuous and linear functional on deformations and that $I$ defined in \eqref{full-I} is  gradient polyconvex.
Define for $y:\O\to\R^3$ smooth enough the following  functional 
\WE
\begin{align}\label{newfunctional}
J(y):= I(y)-\ell (y)\, ,
\end{align} 
\EOR
where $I$ is as in Definition~\ref{def-gpc}. 

\WE
\begin{proposition}
\label{prop-grad-poly}
Let $\O\subset\R^3$ be a bounded Lipschitz domain, %$\alpha,\beta>0$ 
and let  $\Gamma=\Gamma_0\cup\Gamma_1$ be a measurable partition of $\Gamma=\partial\O$ with $\mathcal{H}(\Gamma_0)>0$. Let further $\ell:W^{1,p}(\O;\R^3)\to\R$ be a linear bounded functional and $J$ as in \eqref{newfunctional} with 
\begin{align}\label{part-I}
I(y):= \int_\Omega \hat{W}(\nabla y, \nabla [\mathrm{Cof}\nabla y]) \d x
\end{align}being gradient polyconvex and such that \eqref{growth-graddet1} holds true. Let $L:\R^{3\times 3}\to\R$ be lower semicontinuous.  Finally, let  $p\ge 2$, $q\ge\frac{p}{p-1}$, $r>1$,  $s>0$ and assume that for some given map $y_0\in W^{1,p}(\O;\R^3)$ the following set %Assume that $W(F, \cdot, \cdot)$ is convex for any $F\in \R^{3\times 3}$ and that \eqref{growth-graddet1} and\eqref{pfi-2} hold true. 
 \begin{align*}
\mathcal{A}:&=\{y\in W^{1,p}(\O;\R^3):\ \cof \nabla y\in W^{1,q}(\O;\R^{3\times 3}),\ \det \nabla y\in L^r(\O),\nonumber\\
& \qquad (\det\nabla y)^{-s}\in L^1(\O),\  \det\nabla y>0\mbox{ a.e.\ in $\Omega$},\, L(\nabla y)\le 0\mbox{ a.e.\ in $\Omega$},\ y=y_0\ \mbox{ on }  \Gamma_0\}
\end{align*}
is nonempty and that $\inf_{\mathcal{A}} J<+\infty$. Then the following holds:

\noindent (i) The functional  $J$  has a minimizer on $\mathcal{A}$, i.e., $\inf_{\mathcal{A}} J$ is attained.

\noindent (ii) Moreover, if $q>3$ and $s>6q/(q-3)$ then there is $\varepsilon>0$ such that for every minimizer  $\tilde y\in \mathcal{A}$  of $J$ it holds that $\det\nabla \tilde y\ge\varepsilon$ in $\bar\O$.
\end{proposition}

Before continuing with the proof, we recall  that  the cofactor  of an invertible matrix $F\in\R^{3\times 3}$  consists of all nine  $2\times 2$ subdeterminants of $F$; cf.~\eqref{cofac}. If $A\in\R^{2\times 2}$ and $|A|$ denotes the Frobenius norm of $A$, then 
the Hadamard inequality implies 
\begin{align}\label{Hadamard}
|\det A|\le\frac{|A|^2}{2} .
\end{align}
Applying \eqref{Hadamard} to all nine  $2\times 2$ submatrices of $F\in\R^{3\times 3}$ we get  
\begin{align}\label{h-cof}
|\cof F|\le \frac{3}{2}|F|^2 .
\end{align}
Since 
\begin{align}\label{cof}
\cof F^{-1}=\frac{ F^\top}{\det F}=(\cof F)^{-1} ,
\end{align}
we have 
\begin{align}\label{inverse}
|(\cof F)^{-1}|\le \frac32|F^{-1}|^2 .
\end{align}

It follows from \eqref{detcof-n} that if $F\in\R^{3\times 3}$ is invertible then 
\begin{align}\label{detofcof}
\det\cof F =\det^{\!\!2} F=\det F^2\ , 
\end{align}
where $\det^{\!\! 2}F:=(\det F)^2$.
Finally, we recall (cf.~\cite[Proposition 2.32]{dacorogna}, for instance) that $F\mapsto\det F$ is  locally Lipschitz and  that there is $d>0$ such that for every $F_1,F_2\in\R^{3\times 3}$
\begin{align}\label{n-Lipschitz}
|\det F_1-\det F_2|\le d(1+|F_1|^2+|F_2|^2)|F_1-F_2|\ .
\end{align}

\begin{proof}
We first prove the claim (i).
 Let $\{y_k\}\subset\mathcal{A}$ be a minimizing sequence of $J$. We get that 
\begin{align}\label{bound} &\sup_{k\in\N}\big(\|y_k\|_{W^{1,p}(\O;\R^3)} +\|\cof\nabla y_k\|_{W^{1,q}(\O;\R^{3\times 3})}\nonumber\\
&+\|\det\nabla y_k\|_{L^r(\O)}+\|(\det\nabla y_k)^{-s}\|_{L^1(\O)}\big)<C , \end{align}
for some $C>0$ due to coercivity of $\hat W$, \eqref{growth-graddet1}, and Dirichlet boundary conditions on $\Gamma_0$.
Standard results on weak convergence of minors \cite[Theorems~7.6-1 and 7.7-1]{ciarlet} show that (for a non-relabeled) subsequence 
$y_k\wto y$ in $W^{1,p}(\O;\R^3)$, $\cof\nabla y_k\wto\cof\nabla y$ in $L^q(\O;\R^{3\times 3})$ and $\det\nabla y_k\wto\det\nabla y$ in $L^r(\O)$ for $k\to+\infty$. By weak sequential compactness of bounded sets in $W^{1,q}(\O;\R^{3\times 3})$  we also have that 
$\cof\nabla y_k\wto H$ in $W^{1,q}(\O;\R^{3\times 3})$ for some $H\in W^{1,q}(\O;\R^{3\times 3})$. In particular, $\cof\nabla y_k\to H$ in $L^q(\O;\R^{3\times 3})$. But this implies that $H=\cof\nabla y$.
Hence, there is a subsequence (not relabeled) such that 
$\cof\nabla y_k\to\cof\nabla y$ pointwise almost everywhere in $\O$ for $k\to+\infty$. Formula \eqref{detofcof} yields  that $\det\nabla y_k\to\det\nabla y$ 
pointwise almost everywhere in $\O$ for $k\to+\infty$, too.

We must show that $y\in\mathcal{A}$.
Since $\det\nabla y_k >0$ almost everywhere, we have $\det\nabla y\ge 0$ in the limit.  
Moreover, conditions \eqref{growth-graddet1}, \eqref{bound}, and the Fatou lemma imply that 
  $$
  +\infty>\liminf_{k\to+\infty} J(y_k)+\ell(y_k)\ge \liminf_{k\to+\infty}\int_\O\frac{1}{(\det\nabla y_k(x))^s}\,\md x\ge \int_\O\frac{1}{(\det\nabla y(x))^s}\,\md x , $$
hence, inevitably,  $\det\nabla y>0$ almost everywhere in $\O$ and $(\det\nabla y)^{-s}\in L^1(\O)$. Finally, the continuity of the trace operator shows that $y\in\mathcal{A}$.

By \eqref{cofac} we have
$$(\nabla y_k(x))^{-1}=\frac{(\cof\nabla y_k(x))^\top}{\det\nabla y_k(x)} $$
and thus,  for almost all $x\in\O$
\begin{align}\label{inv-conv}(\nabla y_k(x))^{-1}\longrightarrow (\nabla y(x))^{-1}.\end{align}
Notice that, due to \eqref{inverse}, for almost all $x\in\O$
$$|\nabla y_k(x)|= \det\nabla y_k(x)|(\cof(\nabla y_k(x))^{-\top}|\le \frac32\det\nabla y_k(x)|(\nabla y_k(x))^{-1}|^2<C(x)\, $$
for some $C(x)>0$ independent of $k\in\N$, i.e., we may select a ($x$-dependent) convergent subsequence of $\{\nabla y_k(x)\}_{k \in \N}$ called  $\{\nabla y_{k_m}(x)\}_{m \in \N}$.
Moreover, we have due to \eqref{inverse}  for the fixed $x\in\O$ and $m\to+\infty$
\begin{align*}
\nabla y_{k_m}(x)&=\det\nabla y_{k_m}(x)(\cof(\nabla y_{k_m}(x))^{-\top}%\nonumber\\
%&
\longrightarrow \det\nabla y(x) (\cof(\nabla y(x))^{-\top}=\nabla y(x),
\end{align*}
Now, as the limit is the same for all subsequences of $\{\nabla y_k(x)\}_{k\in\N}$, namely $\nabla y(x)$, we get that the whole sequence converges pointwise almost everywhere. This also shows that for almost every $x\in\O$ we have $L(\nabla y(x))\le\liminf_{k\to+\infty}L(\nabla y_k(x))\le 0$ because $L$ is lower semicontinuous. Hence, $y\in\mathcal{A}$.
As the Lebesgue measure of $\O$ is finite, we get by the Egoroff theorem  that $\nabla y_k\to\nabla y$ in measure, see e.g.\ \cite[Theorem~2.22]{fonseca-leoni}.

 Due to continuity and nonnegativity of $\hat W$ and due to convexity of $\hat W(F,\cdot)$  we have by \cite[Corollary~7.9]{fonseca-leoni} that 
 \begin{align}\label{wlsc} \int_\O \hat W(\nabla y(x),\nabla[\cof\nabla y(x)])\,\md x
 \le\liminf_{k\to+\infty}\int_\O \hat W(\nabla y_k(x),\nabla[\cof\nabla y_k(x)])\,\md x .
\end{align}
To pass to the limit in the functional $\ell$, we exploit its linearity.  Hence, $J$ is weakly lower semicontinuous along $\{y_k\} \subset \mathcal{A}$ and   $y\in\mathcal{A}$ is  a minimizer of $J$. This proves (i).

Let us prove (ii).
If $q>3$,
the Sobolev embedding theorem implies that $\cof\nabla y\in C^{0,\alpha}(\bar\O)$, where $\alpha= (q-3)/q< 1$.  We first show  that $\det\nabla y>0$ on $\partial\Omega$. Assume that there is $x_0\in\partial\Omega$ such that $\det\nabla y(x_0)=0$. We can, without loss of generality, assume that $x_0:=0$ and estimate for  all $x\in\bar\O$  
\begin{align*}
0\le|\cof\nabla y(x)-\cof\nabla y(0)|\le  C|x|^\alpha ,
\end{align*}
where  $C>0$. 
Taking into account \eqref{detofcof}, \eqref{n-Lipschitz}, and the fact that $\cof \nabla y$ is uniformly bounded on $\Omega$, we have for some $K,\tilde d>0$ that 
\begin{align}
0\le \det^{\!\!2}\nabla y(x)=|\det\cof\nabla y(x)-\det\cof\nabla y(0)|\le \tilde d|\cof\nabla y(x)-\cof\nabla y(0)|\le K|x|^\alpha\ .
\end{align}
 Altogether, we see that if $|x|\le t$ then 
$$
\frac{1}{(\det^{\!\!2}\nabla y(x))^{s/2}}\ge \frac{1}{ K^{s/2}t^{\alpha s/2}}\ .$$ 
We have for $t>0$ small enough
\begin{align}\label{HK-inequality}
\int_\O\frac{1}{(\det\nabla y(x))^s}\,\md x\ge \int_{B(0,t)\cap\O}\frac{1}{(\det^{\!\!2}\nabla y(x))^{s/2}}\,\md x\ge \frac{ \hat K t^3}{ K^{s/2}t^{\alpha s/2}}\ ,
\end{align}
where $\hat K t^3\le \mathcal{L}^3(B(0,t)\cap\O)$ and $\hat K>0$ is independent of $t$ for $t$ small enough.  Note that this is possible because $\O$ is Lipschitz and therefore it has the cone property.
Passing to the limit for $t\to 0$ in \eqref{HK-inequality} we see that $\int_\O(\det\nabla y(x))^{-s}\,\md x =+\infty$ because  $\alpha s/2=((q-3)/q)s/2>3$. This, however, contradicts our assumptions and, consequently, 
$\det\nabla y>0$ everywhere in $\partial\O$.  If we assume that there is  $x_0\in\O$ such that $\det\nabla y(x_0)=0$ the same reasoning brings us again to a contradiction. It is even easier because $B(x_0,t)\subset\O$ for $t>0$ small enough. 
%An analogous argument shows that $\det\nabla y>0$ everywhere in $\O$.  
As $\bar\O$ is compact and 
$x\mapsto\det\nabla y(x)$ is continuous it is clear that there is $\varepsilon>0$ such that  $\det\nabla y\ge\varepsilon$ in $\bar\O$. This finishes the proof of (ii) if there is a finite number of minimizers.

Assume then that there are infinitely many minimizers of $J$ and assume that such $\varepsilon>0$ does not exist, i.e., that for every $k\in\N$ there existed $y_k\in\mathcal{A}$ such that $J(y_k)=\inf_{\mathcal{A}}J$  and $x_k\in\bar\O$ such that $\det^{\!\!2}\nabla y_k(x_k)<1/k$. \EOR \SECOND We can  even assume that $x_k\in\O$ because of the continuity of $x\mapsto\det^{\!\!2}\nabla y_k(x)$. \EOR   \WE
The uniform (in $k$)  H\"{o}lder continuity of $x\mapsto\det^{\!\!2}\nabla y_k(x)$ which comes from \eqref{growth-graddet1} and from the fact that $J(y_k)=\inf_{\mathcal{A}}J$ implies that $|\det^{\!\!2}\nabla y_k(x)-\det^{\!\!2}\nabla y_k(x_k)|\le K|x-x_k|^\alpha$ (with $K$ independent of $k$)  and therefore 
$0\le\det^{\!\!2}\nabla y_k(x)\le 1/k+K|x-x_k|^\alpha$. Take $r_k>0$ so small that $B(x_k,r_k)\subset\O$ for all $k\in\N$. 
Then 
$$\inf_{\mathcal{A}}J\ge \int_\O(\det^{\!\!2}\nabla y_k(x))^{-s/2}\,\md x\ge \int_{B(x_k,r_k)}(\det^{\!\!2}\nabla y_k(x))^{-s/2}\,\md x\ge \int_{B(x_k,r_k)}\left(\frac{k}{1+Kkr_k^\alpha}\right)^{s/2}\,\md x . $$
The right-hand side is, however, arbitrarily large for $k$ suitably large and $r_k$ suitably small because $\alpha s>6$. This contradicts $\inf_{\mathcal{A}}J<+\infty$. The proof of (ii) is finished.

\end{proof}

\begin{remark}
 A claim  analogous to (ii)  has already been proved by Healey and Kr\"omer \cite{healey-kroemer}, who showed it for deformations in $W^{2,p}(\Omega; \R^3)$ with $p$ large enough so that the determinant is H\"older continuous. We can ensure the latter even for gradient polyconvex materials, i.e., even though the  integrability of second derivatives of the deformation is not guaranteed.
\end{remark}

\begin{remark}
\noindent 
(i) Using \eqref{detofcof}, \eqref{n-Lipschitz}, and \cite[Thm.~1]{marcus-mizel} we obtain that under the assumptions of Proposition~\ref{prop-grad-poly} every $y\in\mathcal{A}$ satisfies
$\det^{\!\!2}\nabla y=\det\cof\nabla y\in W^{1,r}(\O)$ where 
$$
r:=
\begin{cases}
3q/(9-2q) & \text{ if $9/5\le q<3$,}\\
q & \text{ if $q>3$.}
\end{cases}
$$
One cannot, however, infer any Sobolev regularity of $\det\nabla y$ from this result because $a\mapsto\sqrt{a}$ is not locally Lipschitz on nonnegative reals; cf.~\cite[Thm.~1]{marcus-mizel} for details.\\

\WE

\noindent
(ii) Having $y\in W^{2,p}(\O;\R^3)$ we get from  \cite[Thm.~1]{marcus-mizel} that $\cof\nabla y\in W^{1,q}(\O;\R^{3\times 3})$ where 

$$
q:=
\begin{cases}
3p/(6-p) & \text{ if $3/2\le p< 3$,}\\
p & \text{ if $p>3$.}
\end{cases}
$$

\end{remark}

\EOR

\begin{remark}
(i) The nonemptiness of $\mathcal{A}$ must be explicitly assumed because a precise characterization of the set of traces of Sobolev maps with positive determinant almost everywhere in $\O$ is not known. Thus, given a boundary datum $y_0$ there is no more explicit way to assure that it is compatible with deformations of finite energy than the assumption that $\mathcal{A}$ is non-empty. 
%Some partial results can be found in   \cite{benesova-kamp,benesova-kruzik}. 
We also refer to \cite[Section~7]{benesova-kruzik-wlsc} for more details on this topic.

\noindent 
(ii) Let us point out that convexity of $\hat W$  in its first component is not required in the existence result in Proposition \ref{prop-grad-poly}. This is, on one hand, not surprising since the energy depends (in parts) also on the second gradient. On the other hand, we know that boundedness of the second gradient may not be assured so that no compact embedding can be used to see that the minimizing sequence actually converges strongly in $W^{1,p}(\Omega;\R^3)$ to pass to the limit in the terms depending on $\nabla y(x)$. Actually our proof relies only on pointwise convergence that can be deduced from Cramer's rule  for the  matrix inverse, cf.\ \eqref{cofac}, \WE so that it combines analytical and  algebraic results.\EOR
\end{remark}

\SECOND

The technique of the proof of Proposition \ref{prop-grad-poly} can actually be used  to show the following strong compactness result which might be of an independent interest. Different variants of the proposition are certainly  available, too. 

\begin{proposition}\label{compactness}
Let $\O\subset\R^3$, be a Lipschitz bounded domain and let $\{y_k\}_{k\in\N}\subset W^{1,p}(\O;\R^3)$ for  $p>3$ be such that for some $s>0$
\begin{align}\label{sup}
\sup_{k\in\N}\,\left(\|y_k\|_{W^{1,p}(\O;\R^3)}+\|\cof \nabla y_k\|_{\mathrm{BV}(\O;\R^{3\times 3})}+\||\det \nabla y_k|^{-s}\|_{L^1(\O)}\right)<+\infty .
\end{align}
Then there is a (nonrelabeled) subsequence and $y\in W^{1,p}(\O;\R^3)$ such that for $k\to+\infty$ we have the following convergence results:
$y_k\to y$ in $W^{1,d}(\O;\R^3)$  for every $1\le d<p$, $\det\nabla y_k\to\det\nabla y$ in $L^{r}(\O)$ for every $1\le r<p/3$,  $\cof\nabla y_k\to\cof\nabla y$ in $L^q(\O;\R^{3\times 3})$ for every $1\le q< p/2$, and $|\det\nabla y_k|^{-t}\to |\det\nabla y|^{-t}$  in $L^1(\O)$ for every  $0\le t<s$. Moreover, if $s>3$ then $(\nabla y_k)^{-1}\to (\nabla y)^{-1}$ in $L^{\alpha}(\O;\R^{3\times 3})$ for every  $1\le\alpha< 3s/(3s+3-s)$.  
\end{proposition}

In the above proposition, $\mathrm{BV}(\O;\R^{3\times 3})$ stands for the space of functions of  bounded variations.

\begin{proof}
Reflexivity of $W^{1,p}(\O;\R^3)$ and \eqref{sup} imply the existence of a subsequence of $\{y_k\}$ (which we do not relabel) such that $y_k\wto y$ in  $W^{1,p}(\O;\R^3)$. Moreover, by the compact embedding of $\mathrm{BV}(\O)$ to $L^1(\O)$, we extract a further subsequence satisfying 
$\cof\nabla y_k\to H$ in $L^1(\O;\R^{3\times 3})$. Weak continuity of $y\mapsto\det\nabla y :W^{1,p}(\O;\R^3)\to L^{p/3}(\O)$ and of  $y\mapsto\cof\nabla y :W^{1,p}(\O;\R^3)\to L^{p/2}(\O;\R^{3 \times 3})$ \cite{dacorogna} implies that $H=\cof\nabla y$.  The strong convergence in $L^1$ yields that we can extract a further subsequence ensuring $\cof\nabla y_k \to\cof\nabla y$ a.e.~in $\O$ and in view of \eqref{detcof-n} also 
$\det\nabla y_k\to\det\nabla y$ a.e.~ in $\O$. By the Fatou lemma and the assumption \eqref{sup}
$$
+\infty >\liminf_{k\to+\infty}\int_\O\frac1{|\det\nabla y_k(x)|^s}\,\md x\ge \int_\O\frac1{|\det\nabla y(x)|^s}\,\md x ,$$
which shows that $\det\nabla y\ne 0$ almost everywhere in $\O$. Reasoning analogously to the one in the proof of Proposition~\ref{prop-grad-poly} results in 
the almost everywhere convergence $\nabla y_k\to\nabla y$ for $k\to+\infty$. The  Vitali convergence theorem (see e.g.~\cite[Theorem~2.24]{fonseca-leoni}) then implies that $\nabla y_k\to\nabla y$ in $L^d(\O;\R^{3\times 3})$ for every $1\le d<p$. This shows the strong convergence of 
$y_k\to y$ in $W^{1,d}(\O;\R^3)$. The same argument gives the other  strong convergences of the determinant and the cofactor. The standard embedding  result \cite{ambrosio} implies that $\sup_{k\in\N}\|\cof\nabla y_k\|_{L^{3/2}(\O;\R^{3\times 3})}<+\infty$. As $\{(\det\nabla y_k)^{-1}\}_{k\in\N}\subset L^s(\O)$ is uniformly bounded, too, the H\"{o}lder inequality and the Vitali convergence theorem imply the strong convergence of $(\nabla y_k)^{-1}\to(\nabla y)^{-1}$ in $L^{\alpha}(\O;\R^{3\times 3})$.

\end{proof}

\WE
Let us finally mention that Proposition~\ref{prop-grad-poly}  can be easily generalized to integral functionals 
\begin{align}\label{nf}
I(y):=\int_\O\mathbb{W}(x, y(x),\nabla y(x),(\nabla y(x))^{-1}, \nabla[\cof\nabla y(x)])\,\md x\  
\end{align}
where $$\mathbb{W}:\O\times\R^3\times\R^{3\times 3}\times\R^{3\times 3}\times\R^{3\times 3\times 3}\to \R$$
is  such that $\mathbb{W}$ is a normal integrand (i.e., measurable in $x\in\O$ if the other variables are fixed and 
lower semicontinuous in the other variables if almost every $x\in\O$ is fixed)  
and $\mathbb{W}(x,\hat y, F,F^{-1},\cdot):\R^{3\times 3\times 3}\times\R^{3}\to [0,+\infty]$ is convex for almost every $x\in\O$, every $\hat y\in\R^3$, and every $F\in\R^{3\times 3}$. 
If, further,
for some $c>0$, $s>0$, $p,q,r\ge 1$, and $g\in L^1(\O)$ it holds that for almost all $x\in\O$ and all $\tilde y\in\R^3$, $F\in\R^{3\times 3}$, and $\Delta_1\in\R^{3\times 3\times 3}$  
\begin{align}\label{growth-graddet3}
\mathbb{W}(x,\hat y, F,F^{-1},\Delta_1)\ge\begin{cases}
c\big(|F|^p  +|\cof F|^q+(\det F)^r+ (\det F)^{-s}+|\Delta_1|^q \big)-g(x)
&\text{ if }\det F>0,\\
+\infty&\text{ otherwise} \end{cases}
\end{align}
then we have the following result, which can be phrased not only for the physically most interesting case $n=3$, but for arbitrary dimensions $n\geq 2$. 

\begin{proposition}
\label{prop-grad-poly3}
Let $\O\subset\R^3$ be a bounded Lipschitz domain, %$\alpha,\beta>0$ 
and let  $\Gamma=\Gamma_0\cup\Gamma_1$ be a measurable partition of $\Gamma=\partial\O$ with $\mathcal{H}^2(\Gamma_0)>0$.  Let $I$ be as in \eqref{nf} and such that \eqref{growth-graddet3} holds true with some $g\in L^1(\O)$. Let $L:\R^{3\times 3}\to\R$ be lower semicontinuous.  Finally, let  $p\ge 2$, $q\ge\frac{p}{p-1}$, $r>1$,  $s>0$ and assume that for some given map $y_0\in W^{1,p}(\O;\R^3)$ the following set %Assume that $W(F, \cdot, \cdot)$ is convex for any $F\in \R^{3\times 3}$ and that \eqref{growth-graddet1} and\eqref{pfi-2} hold true. 
 \begin{align*}
\mathcal{A}:&=\{y\in W^{1,p}(\O;\R^3):\ \cof \nabla y\in W^{1,q}(\O;\R^{3\times 3}),\ \det \nabla y\in L^r(\O),\nonumber\\
& \qquad (\det\nabla y)^{-s}\in L^1(\O),\  \det\nabla y>0\mbox{ a.e.\ in $\Omega$},\ \, L(\nabla y)\le 0\mbox{ a.e.\ in $\Omega$},\ y=y_0\ \mbox{ on }  \Gamma_0\}
\end{align*}
is nonempty and that  $\inf_{\mathcal{A}} I<+\infty$. Then the following holds:

\noindent (i) the functional  $I$  has a minimizer on $\mathcal{A}$.

\noindent (ii)  Moreover, if $q>3$ and $s>6q/(q-3)$ then there is $\varepsilon>0$ such that for every minimizer  $\tilde y\in \mathcal{A}$  of $I$ it holds that $\det\nabla \tilde y\ge\varepsilon$ in $\bar\O$.
\end{proposition}

\begin{proof}[Sketch of proof]
The proof follows the lines of the proof of Proposition~\ref{prop-grad-poly}. We can assume that a minimizing sequence $\{y_k\}_{k\in\N}\subset\mathcal{A}$ converges in measure to $y\in\mathcal{A}$ due to the compact embedding of $W^{1,p}(\O;\R^3)$ into  $L^p(\O;\R^3)$. Moreover,    $\{\nabla y_k\}_{k\in\N}$ and $\{(\nabla y_k)^{-1}\}_{k\in\N}$ converge in measure, too; cf.~\eqref{inv-conv}. Then we again apply \cite[Cor.~7.9]{fonseca-leoni} with $v:=\nabla [\cof\nabla y]$, $u:=(y,\nabla y,(\nabla y)^{-1})$, $f:=\mathbb{W}+g$,  $E:=\O$, and correspondingly for the sequences.
\end{proof}
\EOR

\WE 

Let us now turn our attention to gradient polyconvex energies with $\hat{W}$ depending also on the gradient of the Jacobian in a convex way as in \eqref{full-I}. This setting allows for a stronger result with respect to the determinant constraint: Here, we only require the Sobolev index $r$ related to the determinant to be larger than in the existence result but leave the Sobolev index $q$ giving the regularity of the cofactor matrix the same, compare Propositions~\ref{prop-grad-poly} and \ref{prop-grad-poly2}.

\begin{proposition}\label{prop-grad-poly2}
Let $\O\subset\R^3$ be a bounded Lipschitz domain, %$\alpha,\beta>0$ 
and let  $\Gamma=\Gamma_0\cup\Gamma_1$ be a measurable partition of $\Gamma=\partial\O$ with $\mathcal{H}(\Gamma_0)>0$. Let further $\ell:W^{1,p}(\O;\R^3)\to\R$ be a linear bounded functional and $J$ as in \eqref{newfunctional} with $I$ being gradient polyconvex and such that \eqref{growth-graddet2} holds true. Let $L:\R^{3\times 3}\to\R$ be lower semicontinuous.  Finally, let  $p\ge 2$, $q\ge\frac{p}{p-1}$, $r>1$,  $s>0$ and assume that for some given map $y_0\in W^{1,p}(\O;\R^3)$ the following set %Assume that $W(F, \cdot, \cdot)$ is convex for any $F\in \R^{3\times 3}$ and that \eqref{growth-graddet1} and\eqref{pfi-2} hold true. 
 \begin{align*}
\mathcal{A}:&=\{y\in W^{1,p}(\O;\R^3):\ \cof \nabla y\in W^{1,q}(\O;\R^{3\times 3}),\ \det \nabla y\in W^{1,r}(\O),\nonumber\\
& \qquad (\det\nabla y)^{-s}\in L^1(\O),\  \det\nabla y>0\mbox{ a.e.\ in $\Omega$},\, L(\nabla y)\le 0\mbox{ a.e.\ in $\Omega$},\ y=y_0\ \mbox{ on }  \Gamma_0\}
\end{align*}
is nonempty and that $\inf_{\mathcal{A}} J<+\infty$. Then the following holds:

\noindent (i) The functional  $J$  has a minimizer on $\mathcal{A}$.

\noindent (ii) Moreover, if $r>3$ and $s>3r/(r-3)$ then there is $\varepsilon>0$ such that for every minimizer  $\tilde y\in \mathcal{A}$  of $J$ it holds that $\det\nabla \tilde y\ge\varepsilon$ in $\bar\O$.
\end{proposition}

\begin{proof}[Sketch of proof]
We proceed exactly as in the proof of Proposition~\ref{prop-grad-poly}. If $\{y_k\}_{k\in\N}\subset\mathcal{A}$ is a minimizing sequence for $J$ we get by 
standard arguments \cite{ball77,ciarlet} that $\det\nabla y_k\wto\det\nabla y$ in $L^r(\O)$ for $k\to+\infty$. At the same time, we can assume that  $\det\nabla y_k\wto\delta$ in $W^{1,r}(\O)$ for some $\delta\in W^{1,r}(\O)$. These two facts imply that $\delta=\det\nabla y$.  This shows (i). 
In order to prove (ii) we again argue by a contradiction.  Assume that there is $x_0\in\bar{\O}$ such that $\det\nabla y(x_0)=0$.  Let again $x_0:=0$, so that $\det\nabla  y(0)=0$. Then we get from the H\"{o}lder estimate for almost every $x\in\O$ 
\begin{align}
0\le\det\nabla y(x)=|\det\nabla y(x)-\det\nabla y(0)|\le K|x|^\alpha\ ,
\end{align}
where $\alpha=(r-3)/r$. Similar reasoning as in \eqref{HK-inequality} implies the result.

\end{proof}

\EOR

\WE
\begin{remark}[Global invertibility of deformations] 
A global  (a.e.)\ invertibility condition, the so-called Ciarlet-Ne\v{c}as condition  \cite{ciarlet-necas2}, 
\begin{align}\label{c-n}
\int_\O\det\nabla y(x)\,\md x\le\mathcal{L}^3(y(\O))\ 
\end{align}
can be easily imposed on the minimizer if $p>3$  and the proof proceeds exactly in the same way as in \cite{ciarlet-necas2}. This means that there is $\omega\subset\O$, $\mathcal{L}^3(\omega)=0$, such that  $y:\O\setminus\omega\to y(\Omega\setminus\omega)$ is injective.  As $y$ satisfies Lusin's N-condition, we also have that $\mathcal{L}^3(y(\Omega\setminus\omega))=\mathcal{L}^3(y(\Omega))$.  If $|\nabla y|^3/\det\nabla y\in L^{2+\delta}(\O)$
for some $\delta>0$ and \eqref{c-n} holds  then we even get invertibility everywhere in $\O$ due to \cite[Theorem~3.4]{hencl-koskela}. Namely,  this then implies that $y$ is an open map. Moreover, 
the mapping $F\mapsto |F|^3/\det F$ is even polyconvex and positive if $F\in\R^{3\times 3}$ and $\det F>0$, hence the term $\int_\O (|\nabla y(x)|^3/\det\nabla y(x))^{2+\delta}\,\md x$  can be easily added to the energy functional and it  preserves its weak lower semicontinuity, see \cite{ball77, ciarlet} for details.

\end{remark}
%%%%%%%%%%%%%%%%%%%%%%%%%%%%%%
{\bf Open problem.}
Is it possible to construct $y$ such that $y\in \mathcal{A}$ and $y \not\in W^{2,1}(\O;\R^3)$ if  $r$ and $s$ are as   in  (ii) of  Proposition~\ref{prop-grad-poly2} and $\det\nabla y>0$ in $\bar\O$ ?

A partial negative answer can be obtained in the following case: consider  $\O\subset\R^3$  a bounded Lipschitz domain, $F:\O\to\R^{3\times 3}$  such that $\cof F\in W^{1,q}(\O;\R^{3\times 3})$ for $3>q\ge 3/2$, $\det F\ge\varepsilon>0$ in $\O$ for some $\varepsilon>0$, 
and $\det F\in W^{1,r}(\O)$ for some  $r>3$. Then  $F\in W^{1,3q/(6-q)}(\O;\R^{3\times 3})$. Moreover, if  $3<q$ then $F\in W^{1,q}(\O;\R^{3\times 3})$.
Indeed, as $\det F\ge\varepsilon$ we get that $1/\det F=h(\det F)$, where $h:\R \to (0,\infty)$ with
$$h(a):=\begin{cases}
1/|a| & \text{ if $|a|\ge\varepsilon$,}\\
1/\varepsilon & \text{ otherwise.} 
\end{cases}
$$
Since $h$ is a Lipschitz function, we get by  \cite[Thm.~1]{marcus-mizel} that 
$(\det F)^{-1}\in W^{1,r}(\O)$.  In view of \eqref{cofac}, \cite[Thm.~1]{valent} and $\cof F\in W^{1,q}(\O;\R^{3\times 3})$ we get that $F^{-1}\in W^{1,q}(\O;\R^{3\times 3})$. 
Applying \eqref{cofac} to $F^{-1}$  we get that 
$F=\det F(\cof F^{-1})^\top$ whose regularity again follows from \cite[Thm.~1]{marcus-mizel} \cite[Thm.~1]{valent}.
It remains to apply the previous reasoning   to $F:=\nabla y$ to show that minimizers obtained in (ii) of  Propositions~\ref{prop-grad-poly3} and \ref{prop-grad-poly2} satisfy $y\in W^{2,\min(p^*,3q/(6-q))}(\O;\R^3)$ if $3/2\le q<3$ or $y\in W^{2,\min(p^*,q)}(\O;\R^3)$ if $q>3$. In both cases: 
$$p^*:=\begin{cases}
3p/(3-p) & \text{ if $p<3$,}\\
\text{arbitrary number $\ge 1$}  & \text{ if $p=3$,}\\
+\infty & \text{ if $p>3$.}\\
\end{cases}
$$
\EOR
%%%%%%%%%%%%%%%%%%%%%%%%%%

\EOR

\EOR
% \begin{openproblem}
% Is it possible to construct $y$ such that $y\in \mathcal{A}$ and $y \not\in W^{2,1}(\O;\R^3)$ if  $q$ and $s$ are as   in  (ii) of  Proposition~\ref{prop-grad-poly} and $\det\nabla y>0$ in $\bar\O$ ?
% \end{openproblem}

In the next example we provide a modification of the well-known Saint Venant-Kirchhoff material that is gradient polyconvex. 

\begin{example} 
Let $\varphi:\R^{3\times 3}\to\R$ be a stored energy density of an anisotropic Saint Venant-Kirchhoff material, i.e.,  
$$0\le \varphi(F):=\frac{1}{8}\mathcal{C}(F^\top F-\rm Id):(F^\top F-\rm Id) , $$
where $\mathcal{C}$ is the fourth-order and positive definite tensor of elastic constants,   and ``:'' denotes the scalar product between matrices.  Therefore, $\varphi(F)\ge c(|F|^4-1)$ for some $c>0$ and all matrices $F$ and $\varphi(F)=0$ if and only if $F=Q$ for some (not necessarily proper)  rotation $Q$. \FIRST
\MAY We define 
\begin{align}\label{stvk}
\hat W(F,\Delta_1):=\begin{cases}
\varphi(F)+\alpha(|\Delta_1|^q+(\det F)^{-s}) &\text{ if $\det F>0$,}\\ 
+\infty & \text{ otherwise}
\end{cases}
\end{align}
for $F\in\R^{3\times 3}$ and $\Delta_1\in\R^{3\times 3\times 3}$, and for some $\alpha>0$, $s>0$, and  $q=2$. \EOR Then $\hat W$ is admissible in Definition~\ref{def-gpc}, and Proposition~\ref{prop-grad-poly} can be readily  applied with $p=4$. We emphasize that  $\varphi$ is widely used in engineering/computational community because it allows for an easy implementation of all elastic constants. On the other hand, $\varphi$ is not quasiconvex which means that existence of minimizers cannot be guaranteed. Additionally,   $\varphi$ stays locally bounded even  in the vicinity of non-invertible matrices and thus allows for non-realistic material behavior. The function  $\hat W$ cures the mentioned drawbacks and 
\eqref{stvk} offers a mechanically relevant  alternative to the Saint Venant-Kirchhoff model. Indeed, $\hat W$ does not admit any change of the orientation and \MAY $\hat W(F,\Delta_1) \to +\infty$ if $\det F\to 0_+$. \EOR   Obviously, other gradient-polyconvex  variants of \eqref{stvk} are possible, too. 
\EOR
\end{example}

%As before, also this result can be phrased not only for the physically most interesting case $n=3$, but for arbitrary dimensions $n\geq 2$. 

\bigskip

Let us remark that Propositions~\ref{prop-grad-poly}, \ref{prop-grad-poly3} and \ref{prop-grad-poly2} can be further used  to get existence results for  static  as well as for rate-independent evolutionary problems describing   shape memory materials,   elasto-plasticity with non-quasiconvex energy densities, for instance, or for their mutual interactions as in  \cite{GS} or \cite{kruzik-zimmer}, respectively. The main idea of the proofs remains unchanged, however. 
% To get mere existence of a solution one can even weaken bounds on cofactor and determinant in Proposition~\ref{prop-grad-poly}. Namely, it is enough to bound them in the space of maps with  bounded variations because it still ensures   a compact embedding to $L^1$. This  approach enables us to construct simpler  numerical approximations by piecewise affine and globally continuous elements. 
We also refer to \cite{mielke-roubicek1} for a thorough review of mathematical results on rate-independent processes with many applications to materials science.

\bigskip

{\bf Acknowledgment:}   This work was partly conducted when MK held the Giovanni-Prodi Chair in the Institute of Mathematics, University of W\"{u}rzburg. The invitation, hospitality, and support  are  gratefully acknowledged. We thank Jan Valdman for providing us with Figure~\ref{figure1}, Nicola Fusco for mentioning his interest in the general functional \eqref{nf},  Stefan Kr\"{o}mer for interesting discussions, \WE and two anonymous referees for very valuable comments. \EOR  This research  was also supported by the GA\v{C}R grant  17-04301S, by the DAAD--AV\v{C}R project DAAD 16-14 and PPP 57212737 with funds from BMBF.

% Dear Professors Bellomo and Brezzi, \\
% attached, please, find a revised version of our manuscript 01-17-427  submitted to your M3AS.  We thank both reviewers for stimulating and useful suggestions which, we believe, improved the exposition of our work. We have tried to incorporate all of their suggestions to the manuscript. Changes made due to report 1 are in blue, changes made as the feedback of report 2 are in red to ease the reading. Other major changes are marked in magenta.  We want to stress the following points:
% \begin{itemize}

% \item We emphasize that our relaxation result in Section 3 satisfy the usual definition of relaxation. In particular, recovery sequences are constructed. See Remarks 3.3 for details. 
% \item Results of Section 5 are  now stronger than in the original version of our manuscript. In particular, we realized that it is enough to control the cofactor of the deformation gradient in a Sobolev space to obtain existence of minimizers. See e.g.\ Proposition 5.1.   We  believe that the whole section is easier to read now. 
% \item We also changed the definition of ``gradient polyconvexity''  and explained that our existence results  need not only the gradient polyconvexity property but also suitable growth conditions. 

% \end{itemize}

% Besides that, we also changed typos and smaller imperfections that came to our attention. We hope that 
% the current version of the manuscript meets all requirements of the referees.

\end{document}